\documentclass[10pt,dvipdfmx]{article}

\usepackage{latexsym,amsfonts,amsmath,amssymb,mathrsfs,url,amsthm}
\usepackage[dvipdfmx]{color,graphicx}
\newtheorem{theorem}{Theorem}
\newtheorem{lemma}{Lemma}

\newtheorem{remark}{Remark}
\newtheorem{proposition}{Proposition}
\newtheorem{definition}{Definition}
\newtheorem{corollary}{Corollary}
\usepackage[dvipdfmx,
linkbordercolor={1 0 0},
citebordercolor={0 1 0},
urlbordercolor={0 1 1}]{hyperref}

\newcommand{\rd}{\, \mathrm{d}}

\newcommand{\bszero}{\boldsymbol{0}}
\newcommand{\bsone}{\boldsymbol{1}}
\newcommand{\bsalpha}{\boldsymbol{\alpha}}
\newcommand{\bsdelta}{\boldsymbol{\delta}}
\newcommand{\bsc}{\boldsymbol{c}}
\newcommand{\bsd}{\boldsymbol{d}}
\newcommand{\bsk}{\boldsymbol{k}}
\newcommand{\bsl}{\boldsymbol{l}}
\newcommand{\bsr}{\boldsymbol{r}}
\newcommand{\bsq}{\boldsymbol{q}}
\newcommand{\bsx}{\boldsymbol{x}}
\newcommand{\bsy}{\boldsymbol{y}}
\newcommand{\Lcal}{\mathcal{L}}
\newcommand{\Dcal}{\mathcal{D}}
\newcommand{\Scal}{\mathcal{S}}

\newcommand{\wal}{{}_b\mathrm{wal}}
\newcommand{\wor}{\mathrm{wor}}
\newcommand{\EE}{\mathbb{E}}
\newcommand{\FF}{\mathbb{F}}
\newcommand{\NN}{\mathbb{N}}

\newcommand{\RR}{\mathbb{R}}
\newcommand{\VV}{\mathbb{V}}
\newcommand{\ZZ}{\mathbb{Z}}

\allowdisplaybreaks

\title{Recent advances in higher order quasi-Monte Carlo methods}
\author{Takashi Goda\thanks{School of Engineering, University of Tokyo, 7-3-1 Hongo, Bunkyo-ku, Tokyo 113-8656, Japan (\tt{goda@frcer.t.u-tokyo.ac.jp})}, Kosuke Suzuki\thanks{Graduate School of Science, Hiroshima University. 1-3-1 Kagamiyama, HigashiHiroshima, 739-8526, Japan. JSPS Research Fellow. (\tt{kosuke-suzuki@hiroshima-u.ac.jp})}}
\date{\today}

\begin{document}
\maketitle
\begin{abstract}
In this article we review some of recent results on higher order quasi-Monte Carlo (HoQMC) methods. After a seminal work by Dick (2007, 2008) who originally introduced the concept of HoQMC, there have been significant theoretical progresses on HoQMC in terms of discrepancy as well as multivariate numerical integration. Moreover, several successful and promising applications of HoQMC to partial differential equations with random coefficients and Bayesian estimation/inversion problems have been reported recently. In this article we start with standard quasi-Monte Carlo methods based on digital nets and sequences in the sense of Niederreiter, and then move onto their higher order version due to Dick. The Walsh analysis of smooth functions plays a crucial role in developing the theory of HoQMC, and the aim of this article is to provide a unified picture on how the Walsh analysis enables recent developments of HoQMC both for discrepancy and numerical integration.\\[5pt]
\end{abstract}

\section{Introduction}\label{sec:intro}
For an integrable function $f\colon [0,1]^s\to \RR$, we denote the integral of $f$ by
\[ I(f) = \int_{[0,1]^s}f(\bsx) \rd \bsx. \]
Monte Carlo/Quasi-Monte Carlo (QMC) methods are a class of numerical algorithms for approximating $I(f)$ based on pointwise function evaluations. Let $P\subset [0,1]^s$ be a finite multiset, that is, if an element occurs multiple times, it is counted according to its multiplicity. Then $I(f)$ is approximated by
\[ I(f; P) = \frac{1}{|P|}\sum_{\bsx \in P}f(\bsx), \]
where $|P|$ denotes the cardinality of $P$. It is obvious that this algorithm is exact for any choice of $P$ if $f$ is a constant function, but except for such a trivial case, a careful design of $P$ and the accompanying theoretical analysis are required to show that the algorithm works well for various functions $f$.

One fairly easy but sensible approach is to choose each point $\bsx$ independently and uniformly from $[0,1]^s$. This is widely known under the name of Monte Carlo methods \cite{HHbook}. Many fundamental results in probability theory, including the law of large numbers and the central limit theorem, apply to this approach. Looking at $I(f;P)$ as a random variable (with $P$ being the underlying stochastic variable),  we have
\[ \EE[I(f; P)] = I(f)\quad \text{and}\quad \VV[I(f;P)] = \frac{\VV[f]}{|P|},\] 
for any function $f\in L_2([0,1]^s)$, where $\VV[f]$ on the right-hand side of the second equality denotes the variance of $f$. This means, Monte Carlo methods work for any square-integrable functions, but the approximation error converges only probabilistically at the notorious ``one over square root of $N$'' rate. Thus we have a trade-off between versatility and efficiency. 

In some applications where the Monte Carlo convergence is considered too slow, one needs to improve efficiency while discarding versatility of Monte Carlo methods to some extent. QMC methods are one of the standard choices for this purpose. The classical but still central result in QMC methods is the celebrated Koksma-Hlawka inequality:
\begin{align}\label{eq:KH-inequ}
 |I(f; P)-I(f)|\leq V_{\mathrm{HK}}(f) D^*(P), 
\end{align}
where $V_{\mathrm{HK}}(f)$ denotes the total variation of $f$ in the sense of Hardy and Krause, and $D^*(P)$ denotes the star-discrepancy of $P$ (we shall give a precise definition of $D^*(P)$ later in Section~\ref{sec:discrepancy}). Although a class of functions we can deal with is restricted to functions with bounded total variations (i.e., we discard versatility to some extent), through a clever design of $P$ such that $D^*(P)$ is of order better than $|P|^{-1/2}$, the convergence rate can be improved (i.e., we improve efficiency). In fact, there are many explicit constructions of so-called \emph{digital $(t,m,s)$-nets} and \emph{digital $(t,s)$-sequences} achieving the star-discrepancy of order $(\log N)^{s-1}/N$ and $(\log N)^s/N$, respectively,\footnote{To be precise, for an infinite sequence of points $\Scal$, this means that there exists a constant $C_s>0$ depending only on $s$ such that the star-discrepancy of the first $N$ elements of $\Scal$ is bounded by $C_s(\log N)^s/N$ uniformly for all $N$.} see \cite{Nbook,DPbook}. Hence, it follows from the Koksma-Hlawka inequality that the integration error decays faster than the ``one over square root of $N$'' rate.

One natural question in this line is then ``Can we improve efficiency further while sacrificing versatility to more extent?'' Higher order quasi-Monte Carlo (HoQMC) methods due to Dick \cite{Dic07,Dic08} provide an affirmative solution to this question. Now let us focus on functions $f$ having square-integrable partial mixed derivatives up to order $\alpha>1$ in each variable, which obviously means that we discard versatility to more extent than standard QMC methods. In return for this drawback, however, the order of convergence of the integration error can be improved to $(\log N)^{c(\alpha,s)}/N^{\alpha}$ with some exponent $c(\alpha, s)>0$ by employing so-called \emph{higher order digital nets and sequences} as quadrature nodes $P$.\footnote{For higher order digital sequences, this order of convergence does not hold uniformly for all $N$, but holds for a geometric spacing of $N$. It is known that this cannot be improved \cite{Owe16}.} Hence, when the considered integrand is smooth enough, HoQMC methods can be much more efficient than standard QMC methods, not to mention Monte Carlo methods. Of course, one may doubt if there is any chance of encountering with such smooth functions in practice. Fortunately, there have been several successful and promising applications of HoQMC methods reported already in the literature. These include \cite{DKLNS14,DKLS16,DLS16,GHS18a,GHS18b} on applications to partial differential equations with random coefficients (see also the review article \cite{KN16}), and \cite{DGLS17,GP18,DGLS19} on applications to Bayesian estimation/inversion problems. 

Recently there have been significant theoretical progresses on HoQMC methods. The first major step was made by Dick and Pillichshammer \cite{DP14}. They proved that order 5 digital sequences achieve the best possible order of $L_2$-discrepancy, which is $(\log N)^{s/2}/N$, uniformly for all $N$, and moreover, they proved that order 3 digital nets of $N$ points achieve the best possible order of $L_2$-discrepancy, which is $(\log N)^{(s-1)/2}/N$ (here again, we shall give a precise definition of $L_2$-discrepancy later in Section~\ref{sec:discrepancy}). Prior to their work, there had been only one explicit construction of finite point sets (for arbitrarily fixed dimension $s$) with the best possible order of $L_2$-discrepancy due to Chen and Skriganov \cite{CS02,Skr06}. Therefore, higher order digital nets (resp. sequences) are now recognized as the second (resp. first) explicit construction of optimal order $L_2$-discrepancy point sets (resp. sequences). More recently, several refined analyses for generalizing or extending the work of Dick and Pillichshammer have been conducted \cite{Dic14,M15,DHMP17a,DHMP17b,BM18}. 

Another major step was made in a series of papers \cite{GSY1,GSY2,GSY3}, where the authors refined the integration error analysis for smooth functions due to Dick \cite{Dic07,Dic08} and proved that order $(2\alpha+1)$ digital nets and sequences achieve the best possible order of the worst-case error for a reproducing kernel Hilbert space with dominating mixed smoothness $\alpha$, which is $(\log N)^{(s-1)/2}/N^{\alpha}$. Note that the original work by Dick \cite{Dic08} proves the worst-case error of order $(\log N)^{s \alpha}/N^{\alpha}$ for order $\alpha$ digital nets and sequences, see also \cite{BD09}. Other than higher order digital nets and sequences, only the Frolov lattice rule in conjunction with periodization of integrands has been proven to achieve the same, best possible order of the worst-case error so far \cite{Fro76,UU15,NUU17}.

There is a common source for obtaining the result of \cite{DP14} and that of \cite{GSY2,GSY3}, which is the Walsh analysis. To introduce the concept of HoQMC methods originally, Dick managed to prove the decay of the Walsh coefficients of smooth functions \cite{Dic07,Dic08}. In fact, his digit interlacing construction of higher order digital nets and sequences, which shall be described in Section~\ref{subsec:digit-interlacing}, is carefully designed to exploit the decay of the Walsh coefficients. In order to improve his seminal results, it may be sensible to attempt to exploit some further aspect of the Walsh coefficients. Both the result of \cite{DP14} and that of \cite{GSY2,GSY3} rely not only on the decay but also on the sparsity of the Walsh coefficients. 

In this article we mainly focus on the papers \cite{DP14,GSY3} and provide a unified picture on how the Walsh analysis enables recent developments of HoQMC methods both for discrepancy and numerical integration. The rest of this article is organized as follows. In Section~\ref{sec:QMC}, we explain about standard QMC methods based on digital nets and sequences in the sense of Niederreiter \cite{Nbook}. Although integer lattices are another important class of QMC point sets, see for instance \cite{SJ94} and \cite[Section~5]{DKS13}, we do not cover them in this article. In Section~\ref{sec:HoQMC}, we introduce the definitions of higher order digital nets and sequences, and provide an explicit construction algorithm due to Dick \cite{Dic08}. In Section~\ref{sec:Walsh}, we introduce the definition of the Walsh functions and give some key connection to digital nets. Thereafter, recent advances in HoQMC methods for discrepancy are described in Section~\ref{sec:discrepancy}, while those for numerical integration are in Section~\ref{sec:integration}. We shall highlight an analogy between the approach by \cite{DP14} and that of \cite{GSY3}, where exploiting both the decay and the sparsity of the Walsh coefficients plays a crucial role. We conclude the article with some future research directions.\\[5pt]
\indent \textbf{Notation.} Throughout this article, we shall use the following notation. Let $\NN$ be the set of positive integers and we write $\NN_0=\NN\cup \{0\}$. For a prime $b$, let $\FF_b$ be the finite field with $b$ elements and we identify $\FF_b$ with the set of integers $\{0,1,\ldots,b-1\}$ equipped with addition and multiplication modulo $b$. For $x\in [0,1]$, its $b$-adic expansion $x=\sum_{i=1}^{\infty}\xi_ib^{-i}$ with $\xi_i\in \FF_b$ is understood to be unique in the sense that infinitely many of the $\xi_i$'s are different from $b-1$ if $x\neq 1$ and that all of the $\xi_i$'s are equal to $b-1$. Note that for $k=1\in \NN$ we use the $b$-adic expansion $1b^0$, whereas for $x=1\in [0,1]$ we use $(b-1)b^{-1}+(b-1)b^{-2}+\cdots$. It will be clear from the context which expansion we use. The operator $\oplus$ denotes the digitwise addition modulo $b$, that is, for $x,y\in [0,1]$ with $b$-adic expansions given by $x=\sum_{i=1}^{\infty}\xi_ib^{-i}$ and $y=\sum_{i=1}^{\infty}\eta_ib^{-i}$, respectively, $\oplus$ is defined as
\[ x\oplus y = \sum_{i=1}^{\infty}\zeta_ib^{-i}, \quad \text{where $\zeta_i=\xi_i+\eta_i \pmod b$}. \]
Similarly we use $\oplus$ for digitwise addition for non-negative integers based on the $b$-adic expansions. In case of vectors in $[0,1]^s$ or $\NN_0^s$, the operator $\oplus$ is applied componentwise.

\section{Standard quasi-Monte Carlo}\label{sec:QMC}

\subsection{Digital nets and sequences}
We start with a general construction scheme for a class of QMC point sets called \emph{digital nets} due to Niederreiter \cite{Nbook}. Note that both of digital $(t,m,s)$-nets and higher order digital nets can be regarded as special subclasses of them.
\begin{definition}[Digital nets]
Let $m,n\in \NN$ and let $C_1,\ldots,C_s$ be $n\times m$ matrices over $\FF_b$. For an integer $0\leq h<b^m$ with $b$-adic expansion $h = \eta_0 + \eta_1 b+\cdots + \eta_{m-1}b^{m-1}$, define the point $\bsx_h=(x_{h,1},\ldots,x_{h,s})\in [0,1]^s$ by
\[ x_{h,j} = \frac{\xi_{1,h,j}}{b}+\frac{\xi_{2,h,j}}{b^2}+\cdots + \frac{\xi_{n,h,j}}{b^n}, \]
where
\[ (\xi_{1,h,j},\xi_{2,h,j},\ldots, \xi_{n,h,j}) = (\eta_0,\eta_1,\ldots,\eta_{m-1})\cdot C_j^{\top}. \]
The set $P=\{ \bsx_h \mid 0\leq h<b^m\}\subset [0,1]^s$ is called a digital net over $\FF_b$ (with generating matrices $C_1,\ldots,C_s$). 
\end{definition}
\noindent It is obvious from the definition that the parameter $m$ determines the total number of points, which is $b^m$, while the parameter $n$ determines the precision of points. We can extend this definition to construct infinite sequences of points called \emph{digital sequences}. Again, both of digital $(t,s)$-sequences and higher order digital sequences can be regarded as special subclasses of them.

\begin{definition}[Digital sequences]
Let $C_1,\ldots,C_s$ be $\NN\times \NN$ matrices over $\FF_b$. For an integer $h\in \NN_0$ with $b$-adic expansion $h = \eta_0 + \eta_1 b+\cdots$, where all but a finite number of $\eta_i$ are 0, define the point $\bsx_h=(x_{h,1},\ldots,x_{h,s})\in [0,1]^s$ by
\[ x_{h,j} = \frac{\xi_{1,h,j}}{b}+\frac{\xi_{2,h,j}}{b^2}+\cdots , \]
where
\[ (\xi_{1,h,j},\xi_{2,h,j},\ldots ) = (\eta_0,\eta_1,\ldots)\cdot C_j^{\top}. \]
The sequence of points $\Scal=\{ \bsx_h \mid h\in \NN_0\}\subset [0,1]^s$ is called a digital sequence over $\FF_b$ (with generating matrices $C_1,\ldots,C_s$). 
\end{definition}

\begin{remark}\label{rem:precision}
Assume that for each $1\leq j\leq s$ there exists a function $K_j\colon \NN\to \NN$ such that $C_j=(c_{k,l}^{(j)})_{k,l\in \NN}$ satisfies $c_{k,l}^{(j)}=0$ whenever $k>K_j(l)$. Then for any $m\in \NN$ the first $b^m$ elements of a digital sequence  over $\FF_b$ with generating matrices $C_1,\ldots,C_s$ can be identified with a digital net  over $\FF_b$ with generating matrices $C_1^{[n\times m]},\ldots,C_s^{[n\times m]}$ with the precision parameter
\[ n = \max_{1\leq j\leq s}\max_{1\leq l\leq m}K_j(l), \]
where $C_j^{[n\times m]}$ denotes the upper-left $n\times m$ submatrix of $C_j$.
\end{remark}

\subsection{Quality measure}\label{subsec:standard_def}
In order to generate point sets or sequences from the above construction scheme such that the star-discrepancy is small, we need to design generating matrices $C_1,\ldots,C_s$ properly. In this subsection, we introduce the widely-used quality measure called \emph{$t$-value}, which is based on the Niederreiter-Rosenbloom-Tsfasman (NRT) weight function \cite{Nie86,RT97}.

\begin{definition}[Dual nets]
Let $m,n\in \NN$ and let $P$ be a digital net over $\FF_b$ with generating matrices $C_1,\ldots,C_s\in \FF_b^{n\times m}$. The dual net of $P$, denoted by $P^{\perp}$, is defined by
\[ P^{\perp} := \left\{ \bsk=(k_1,\ldots,k_s)\in \NN_0^s \mid C_1^{\top}\nu_n(k_1)\oplus \cdots \oplus C_s^{\top}\nu_n(k_s)=\bszero \in \FF_b^m \right\}, \]
where 
\[ \nu_n(k)=(\kappa_0,\ldots,\kappa_{n-1})^{\top}\in \FF_b^n \]
for $k\in \NN_0$ with $b$-adic expansion $k=\kappa_0+\kappa_1 b+\cdots$, where all but a finite number of $\kappa_i$ are 0.
\end{definition}

\begin{remark}
Let $\Scal$ be a digital sequence over $\FF_b$ for which there exist functions $K_j\colon \NN\to \NN$ such that $C_j=(c_{k,l}^{(j)})_{k,l\in \NN}$ satisfies $c_{k,l}^{(j)}=0$ whenever $k>K_j(l)$. Then the dual net can be defined for the first $b^m$ elements of $\Scal$ for any $m\in \NN$, since they can be identified with a digital net as discussed in Remark~\ref{rem:precision}.
\end{remark}

\begin{definition}[NRT weight function]
For $k\in \NN$, we denote the $b$-adic expansion of $k$ by
\[ k=\kappa_1 b^{c_1-1}+\kappa_2 b^{c_2-1}+\cdots + \kappa_v b^{c_v-1} \]
with $\kappa_1,\ldots,\kappa_v\in \FF_b\setminus \{0\}$ and $c_1>\cdots >c_v>0$. Then the NRT weight function $\mu_1\colon \NN_0\to \NN_0$ is defined by $\mu_1(0)=0$ and $\mu_1(k) = c_1$. In case of vectors in $\NN_0^s$, we define
\[ \mu_1(k_1,\ldots,k_s) = \sum_{j=1}^{s}\mu_1(k_j). \]
\end{definition}

We are ready to introduce the definition of $t$-value.
\begin{definition}[$t$-value]
Let $m,n\in \NN$ and let $P$ be a digital net over $\FF_b$ with generating matrices $C_1,\ldots,C_s\in \FF_b^{n\times m}$. We write
\[ \mu_1(P^{\perp}) := \min_{\bsk\in P^{\perp}\setminus \{\bszero\}}\mu_1(\bsk). \]
Then the $t$-value of $P$ is defined by 
\[ t:=m - \mu_1(P^{\perp}) +1, \]
and $P$ is called a \emph{digital $(t,m,s)$-net} over $\FF_b$.
\end{definition}
\noindent This definition of $t$-value is based on the concept of duality theory of digital nets as originally studied in \cite{NP01}. There is another but equivalent definition of $t$-value: let $\rho$ be the largest integer such that, for any choice $d_1,\ldots,d_s\in \NN_0$ with $d_1+\cdots +d_s=\rho$,\\[5pt]
\indent the first $d_1$ row vectors of $C_1$\\
\indent the first $d_2$ row vectors of $C_2$\\
\indent $\vdots$\\
\indent the first $d_s$ row vectors of $C_s$\\[5pt]
are linearly independent over $\FF_b$. Then the $t$-value can be also defined by $m-\rho$.

Because of the linear independence of the row vectors of generating matrices, any digital $(t,m,s)$-net over $\FF_b$ has the following equi-distribution property: every $b$-adic elementary box of the form
\[ E = \prod_{j=1}^{s}\left[ \frac{a_j}{b^{c_j}}, \frac{a_j+1}{b^{c_j}}\right)\]
with $c_1,\ldots,c_s\geq 0$, $c_1+\cdots + c_s=m-t$ and $0\leq a_j<b^{c_j}$ for all $j$, whose volume is $b^{t-m} $, contains $b^t$ points exactly. Hence, as the $t$-value is smaller, digital nets are more equi-distributed over $[0,1]^s$. This is why the $t$-value works as a quality measure of digital nets.

For digital sequences, the $t$-value is defined as follows.
\begin{definition}
Let $\Scal$ be a digital sequence over $\FF_b$. $\Scal$ is called a \emph{digital $(t,s)$-sequence} over $\FF_b$ if there exists $t\in \NN_0$ such that the first $b^m$ points of $\Scal$ is a digital $(t,m,s)$-net for any $m\geq t$.
\end{definition}

The following result states that the star-discrepancy of digital $(t,m,s)$-nets and digital $(t,s)$-sequences are of order $(\log N)^{s-1}/N$ and $(\log N)^s/N$, respectively, as mentioned in the first section. We refer to \cite[Theorems~4.10 and 4.17]{Nbook} for the proof.
\begin{theorem}\label{thm:star-disc-bound} The following holds true:
\begin{enumerate}
\item Let $P$ be a digital $(t,m,s)$-net over $\FF_b$. There exists a constant $B^{(1)}_{s,b,t}$ such that the star-discrepancy of $P$ is bounded by
\[ D^*(P) \leq B^{(1)}_{s,b,t}\frac{m^{s-1}}{b^m}. \]
\item Let $\Scal=\{ \bsx_h \mid h\in \NN_0\}$ be a digital $(t,s)$-sequence over $\FF_b$. There exists a constant $B^{(2)}_{s,b,t}$ such that the star-discrepancy of the first $N$ points of $\Scal$ is bounded by
\[ D^*(\{\bsx_0,\ldots,\bsx_{N-1}\}) \leq B^{(2)}_{s,b,t}\frac{(\log N)^{s}}{N}, \]
for any $N\geq 2$.
\end{enumerate}
\end{theorem}

We end this subsection by providing one useful result in analyzing the integration error of QMC rules using digital $(t,m,s)$-nets.
\begin{lemma}\label{lem:digital-net-card} Let $P$ be a digital $(t,m,s)$-net over $\FF_b$. The following holds true:
\begin{enumerate}
\item For $z\in \NN_0$, 
\begin{align*}
 \left| \left\{ \bsk\in P^{\perp}\setminus \{\bszero\}\mid \mu_1(\bsk)=z \right\} \right| \leq \begin{cases} 0 & \text{if $z<\mu_1(P^{\perp})$,} \\ b^{z-\mu_1(P^{\perp})+1}(z+1)^{s-1} & \text{otherwise.} \end{cases} 
\end{align*}
\item For any real $\lambda >1$,
\[ \sum_{\bsk \in P^\perp \setminus \{\bszero\}} b^{-\lambda \mu_1(\bsk)} \leq 2^{s-1}b^{\lambda}\frac{(\mu_1(P^\perp))^{s-1}}{b^{\lambda \mu_1(P^\perp)}} \sum_{z=1}^\infty b^{(1-\lambda)z} z^{s-1}. \]
\end{enumerate}
\end{lemma}
\begin{proof}
In this proof we put $A_z = \left| \left\{ \bsk\in P^{\perp}\setminus \{\bszero\}\mid \mu_1(\bsk)=z \right\} \right|$. It holds that
\[ A_z = \sum_{\substack{z_1,\ldots,z_s\in \NN_0\\ z_1+\cdots+z_s=z}}\left| \left\{ \bsk\in P^{\perp}\setminus \{\bszero\}\mid \mu_1(k_1)=z_1,\ldots,\mu_1(k_s)=z_s \right\} \right| . \]
Following \cite[Lemma~2.2]{Skr06}, the summand is bounded above by
\begin{align*}
\left| \left\{ \bsk\in P^{\perp}\setminus \{\bszero\}\mid \mu_1(k_1)=z_1,\ldots,\mu_1(k_s)=z_s \right\} \right| \leq \begin{cases} 0 & \text{if $z<\mu_1(P^{\perp})$,} \\ b^{z-\mu_1(P^{\perp})+1} & \text{otherwise.} \end{cases}
\end{align*}
Thus we have $A_z=0$ if $z<\mu_1(P^{\perp})$, since each summand is 0. For $z\geq \mu_1(P^{\perp})$, this bound gives
\begin{align*}
A_z & \leq  b^{z-\mu_1(P^{\perp})+1} \sum_{\substack{z_1,\ldots,z_s\in \NN_0\\ z_1+\cdots+z_s=z}}1
 = b^{z-\mu_1(P^{\perp})+1}\binom{z+s-1}{s-1} \\
& = b^{z-\mu_1(P^{\perp})+1}\prod_{j=1}^{s-1}\frac{z+j}{j} 
 \leq b^{z-\mu_1(P^{\perp})+1}(z+1)^{s-1},
\end{align*}
which proves the first assertion of the lemma.

Using the result of the first assertion and then applying the change of variables $z \mapsto z+\mu_1(P)-1$, we have
\begin{align*}
\sum_{\bsk \in P^\perp \setminus \{\bszero\}} b^{-\lambda \mu_1(\bsk)} & = \sum_{z=\mu_1(P^{\perp})}^{\infty}b^{-\lambda z}A_z \leq \sum_{z=\mu_1(P^{\perp})}^{\infty}b^{(1-\lambda) z-\mu_1(P^{\perp})+1}(z+1)^{s-1} \\
& = b^{-\lambda(\mu_1(P^{\perp})-1)}\sum_{z=1}^{\infty}b^{(1-\lambda) z}(z+\mu_1(P^{\perp}))^{s-1} \\
& \leq 2^{s-1}(\mu_1(P^{\perp}))^{s-1}b^{-\lambda(\mu_1(P^{\perp})-1)}\sum_{z=1}^{\infty}b^{(1-\lambda) z}z^{s-1},
\end{align*}
where the last sum over $z$ is finite since $\lambda>1$. Hence we complete the proof.
\end{proof}

\subsection{Explicit constructions}\label{subsec:standard_constr}
Theorem~\ref{thm:star-disc-bound} together with the Koksma-Hlawka inequality \eqref{eq:KH-inequ} gives a motivation to construct digital $(t,m,s)$-nets or digital $(t,s)$-sequences with small $t$-value. In fact, many explicit constructions of digital $(t,s)$-sequences with small $t$-value are already known. Examples are given by Sobol' \cite{Sob67}, Faure \cite{Fau82}, Niederreiter \cite{Nie88}, Tezuka \cite{Tez93}, Niederreiter and Xing \cite{NXbook} as well as many others. Here we give one example from \cite{Nie88,Tez93}. 

Let $p_1,p_2,\ldots\in \FF_b[x]$ be a sequence of distinct monic irreducible polynomials over $\FF_b$ with $\deg(p_1)\leq \deg(p_2)\leq \cdots$. For each $j\in \NN$, let $e_j=\deg(p_j)$ and consider the following Laurent series expansion
\begin{align}\label{eq:laurent}
 \frac{x^{e_j-z-1}}{(p_j(x))^i} = \sum_{l=1}^{\infty}\frac{a^{(j)}(i,z,l)}{x^l}\in \FF_b((x^{-1}))
\end{align}
for integers $i\geq 1$ and $0\leq z<e_j$. Define the matrix $C_j=(c_{k,l}^{(j)})_{k,l\in \NN}$ by
\[ c_{k,l}^{(j)} = a^{(j)}\left( \left\lfloor \frac{k-1}{e_j}\right\rfloor+1, (k-1)\bmod e_j, l\right).  \]
Here we see that the rows of $C_j$ (from upper to lower) correspond to the Laurent series expansions of
\[ \frac{x^{e_j-1}}{p_j(x)},\ldots,\frac{1}{p_j(x)},\frac{x^{e_j-1}}{(p_j(x))^2},\ldots,\frac{1}{(p_j(x))^2},\ldots . \]
Hence, we have $K_j(l)=l$ for any $j,l$ in the light of Remark~\ref{rem:precision}, that is, $c_{k,l}^{(j)}=0$ whenever $k>l$, meaning that $C_j$ is an upper triangular matrix. 

It is straightforward from the definition that this explicit construction of digital sequences is extensible in dimension. The first $s$ matrices $C_1,\ldots,C_s$ generate a digital $(t,s)$-sequence over $\FF_b$ with
\[ t \leq \sum_{j=1}^{s}(e_j-1). \]
We refer to \cite[Theorem~8.2]{DPbook} for the proof of this fact.

\begin{remark}
Some comments are in order:
\begin{enumerate}
\item Replacing the numerator $x^{e_j-z-1}$ of \eqref{eq:laurent} by $x^z$, this construction algorithm is the same as the one originally introduced in \cite{Nie88}, which is nowadays known as \emph{Niederreiter sequences}. For the original Niederreiter sequence, the $t$-value is strictly equal to $\sum_{j=1}^{s}(e_j-1)$ \cite{DN08}. 
\item A generalization of Niederreiter sequences by Tezuka \cite{Tez93} is to use the set of linearly independent polynomials $\{y_{j,i,z}(x)\mid 0\leq z<e_j\}$ over $\FF_b$ for the numerator of \eqref{eq:laurent} rather than the simplest set $\{x^z\mid 0\leq z<e_j\}$. 
\item The Sobol' sequences due to \cite{Sob67} are a subclass of the generalized Niederreiter sequences over $\FF_2$, where the primitive polynomials are used for $p_1,p_2,\ldots$. Recently, Faure and Lemieux \cite{FL16} gave some precise connections between the Sobol' sequences and the generalized Niederreiter sequences.
\end{enumerate}
\end{remark}

\subsection{Polynomial lattice point sets}
We end this section by introducing another important class of digital nets called \emph{polynomial lattice point sets} introduced by Niederreiter \cite{Nie92}. 

\begin{definition}
Let $m\in \NN$, and let $p\in \FF_b[x]$ and $\bsq=(q_1,\ldots,q_s)\in (\FF_b[x])^s$ such that $\deg(p)=m$ and $\deg(q_j)<m$. For $1\leq j\leq s$, consider the Laurent series expansion
\[ \frac{q_j(x)}{p(x)} = \sum_{l=1}^{\infty}\frac{a_l^{(j)}}{x^l}\in \FF_b((x^{-1})) \]
and define the Hankel matrix $C_j=(c_{k,l}^{(j)})_{1\leq k,l\leq m}\in \FF_b^{m\times m}$ by
\[ c_{k,l}^{(j)}= a_{k+l-1}^{(j)}. \]
Then a digital net over $\FF_b$ with these generating matrices $C_1,\ldots,C_s$ is called a \emph{polynomial lattice point set} (with modulus $p$ and generating vector $\bsq$).
\end{definition}
\noindent Indeed, polynomial lattice point sets can be constructed without using generating matrices explicitly. Define the map $v_m\colon \FF_b((x^{-1}))\to [0,1]$ by
\[ v_m\left( \sum_{i=w}^{\infty}a_i x^{-i}\right) := \sum_{i=\max\{1,w\}}^{m}a_ib^{-i}. \]
We identify $h\in \NN_0$, whose finite $b$-adic expansion is given by $h=\eta_0+\eta_1b+\cdots$, with the polynomial over $\FF_b$ given by $h(x)=\eta_0+\eta_1x+\cdots$. Put
\[ \bsx_h = \left( v_m\left(\frac{h(x)q_1(x)}{p(x)}\right),\ldots, v_m\left(\frac{h(x)q_s(x)}{p(x)}\right)\right)\in [0,1]^s. \]
Then a point set $\{\bsx_h\mid 0\leq h<b^m\}$ is nothing but the polynomial lattice point set as defined above.

The modulus $p$ is often chosen to be either the monomial $p(x)=x^m$ or irreducible. The difficulty is in how to choose the generating vector $\bsq$. In particular, for $s\geq 3$, no explicit way for this choice has been known yet. Currently one of the most standard approaches is to recursively choose one component $q_j$ from the set $\{q\in \FF_b[x]\mid \deg(q)<m\}$ which minimizes a considered criterion while the earlier ones $q_1,\ldots,q_{j-1}$ kept unchanged. This greedy algorithm is known as \emph{component-by-component construction} \cite{SR02}. Another well-known approach is to restrict ourselves to the form
\[ \bsq = (1,q,\ldots,q^{s-1})\in (\FF_b[x])^s \]
for $q\in \FF_b[x]$ with $\deg(q)<m$, and then to choose one optimal $q$ with respect to a considered criterion. This algorithm is known as \emph{Korobov construction}. Compared to Korobov construction, component-by-component construction has the advantages that it is extensible in dimension and that the fast algorithm using the fast Fourier transform has been well-established \cite{NC06}. However, neither of both is extensible in the number of points.

\section{Higher order quasi-Monte Carlo}\label{sec:HoQMC}

\subsection{Quality measure}\label{subsec:HO_def}
To introduce the definitions of higher order digital nets and sequences, we start with generalizing the NRT weight function.
\begin{definition}[Dick weight function]
Let $\alpha \in \NN$. For $k\in \NN$ we denote the $b$-adic expansion of $k$ by
\[ k=\kappa_1 b^{c_1-1}+\kappa_2 b^{c_2-1}+\cdots + \kappa_v b^{c_v-1} \]
with $\kappa_1,\ldots,\kappa_v\in \FF_b\setminus \{0\}$ and $c_1>\cdots >c_v>0$. Then the Dick weight function $\mu_\alpha\colon \NN_0\to \NN_0$ is defined by $\mu_\alpha(0)=0$ and 
\[ \mu_\alpha(k) = \sum_{i=1}^{\min(\alpha, v)}c_i. \] 
In case of vectors in $\NN_0^s$ we define
\[ \mu_\alpha(k_1,\ldots,k_s) = \sum_{j=1}^{s}\mu_\alpha(k_j). \]
\end{definition}
\noindent It is obvious that the Dick weight function coincides with the NRT weight function when $\alpha=1$. As a natural generalization of digital $(t,m,s)$-nets and $(t,s)$-sequences based on the NRT weight function, higher order digital nets and sequences due to Dick \cite{Dic07,Dic08} are defined by using the Dick weight function as follows:
\begin{definition}[Higher order digital nets]
Let $\alpha \in \NN$. Let $m,n\in \NN$ and let $P$ be a digital net over $\FF_b$ with generating matrices $C_1,\ldots,C_s\in \FF_b^{n\times m}$. We write
\[ \mu_\alpha(P^{\perp}) := \min_{\bsk\in P^{\perp}\setminus \{\bszero\}}\mu_\alpha(\bsk). \]
Then $P$ is called an \emph{order $\alpha$ digital $(t_{\alpha},m,s)$-net} over $\FF_b$ with 
\[ t_{\alpha}:=\alpha m - \mu_\alpha(P^{\perp}) +1. \]
\end{definition}

\begin{definition}[Higher order digital sequences]
Let $\Scal$ be a digital sequence over $\FF_b$. $\Scal$ is called an \emph{order $\alpha$ digital $(t_{\alpha},s)$-sequence} over $\FF_b$ if there exists $t_{\alpha}\in \NN_0$ such that the first $b^m$ points of $\Scal$ is an order $\alpha$ digital $(t_{\alpha},m,s)$-net for any $m\geq t_{\alpha}/\alpha$.
\end{definition}
\noindent
Obviously, for fixed $\alpha$, the $t_{\alpha}$-value works as a quality measure of order $\alpha$ digital nets and sequences.

Our definition of higher order digital nets is again based on the concept of duality theory of digital nets, and looks different from the original definition by Dick which has an additional parameter $\beta$. When $n\geq \alpha m$, however, by setting $\beta=1$, our definition becomes equivalent to his original definition based on the linear independence of row vectors of generating matrices described below: let $\rho$ be the largest integer such that, for any choice $1\leq d_{j,\nu_j}<\cdots < d_{j,1}\leq n$, where $0\leq \nu_j\leq m$ for all $1\leq j\leq s$ with 
\[ \sum_{j=1}^{s}\sum_{i=1}^{\min(\nu_j,\alpha)}d_{j,i} = \rho ,\]
\indent the $d_{1,\nu_1}, \ldots, d_{1,1}$-th row vectors of $C_1$\\
\indent the $d_{2,\nu_2}, \ldots, d_{2,1}$-th row vectors of $C_2$\\
\indent $\vdots$\\
\indent the $d_{s,\nu_s}, \ldots, d_{s,1}$-th row vectors of $C_s$\\[5pt]
are linearly independent over $\FF_b$. Then a digital net with generating matrices $C_1,\ldots,C_s$ is an order $\alpha$ digital $(t_{\alpha},m,s)$-net over $\FF_b$ with $t_{\alpha}=\alpha m-\rho$. This linear independence of the row vectors of generating matrices ensures that a higher order digital net has a similar geometric equi-distribution property to what is described in Subsection~\ref{subsec:standard_def}, see \cite[Chapter~15.3]{DPbook} for more details.

Here we provide one useful property called \emph{propagation rule} of higher order digital nets and sequences shown in \cite[Theorem~3.3]{Dic07}. We give a different proof which does not rely on the linear independence of generating matrices. 
\begin{lemma}\label{lem:propagation}
For $\beta\in \NN$, $\beta\geq 2$, let $P$ and $\Scal$ be an order $\beta$ digital $(t_{\beta},m,s)$-net and an order $\beta$ digital $(t_{\beta},s)$-sequence over $\FF_b$, respectively. Then, for any $1\leq \alpha< \beta$, $P$ and $\Scal$ are also an order $\alpha$ digital $(t_{\alpha},m,s)$-net over $\FF_b$ and an order $\alpha$ digital $(t_{\alpha},s)$-sequence over $\FF_b$, respectively, both with $t_{\alpha} \leq \lceil t_{\beta} \alpha/\beta \rceil$.
\end{lemma}
\begin{proof}
First we prove that the inequality
\[ \frac{\mu_{\alpha}(\bsk)}{\alpha} \geq \frac{\mu_{\beta}(\bsk)}{\beta} \]
holds for any $\bsk\in \NN_0^s$ and $1\leq \alpha\leq \beta$. Since the weight function for vector $\bsk$ is defined as the sum of the weight function for each coordinate, it suffices to prove the one-dimensional case. Since the result for $k=0$ follows trivially, let us consider $k>0$. Denote the $b$-adic expansion of $k$ by
\[ k=\kappa_1 b^{c_1-1}+\kappa_2 b^{c_2-1}+\cdots + \kappa_v b^{c_v-1} \]
with $\kappa_1,\ldots,\kappa_v\in \FF_b\setminus \{0\}$ and $c_1>\cdots >c_v>0$, and write $c_{v+1}=c_{v+2}=\cdots =0$. Then we have
\[ \frac{\mu_{\alpha}(k)}{\alpha} =\frac{1}{\alpha}\sum_{i=1}^{\alpha}c_i\geq \frac{1}{\beta}\sum_{i=1}^{\beta}c_i=\frac{\mu_{\beta}(k)}{\beta}, \]
which proves the assertion.

Now let us consider an order $\beta$ digital $(t_{\beta},m,s)$-net over $\FF_b$. Using the above inequality we obtain
\[ \mu_\alpha(P^{\perp}) = \min_{\bsk\in P^{\perp}\setminus \{\bszero\}}\mu_\alpha(\bsk) \geq \min_{\bsk\in P^{\perp}\setminus \{\bszero\}}\frac{\alpha}{\beta}\mu_\beta(\bsk) = \frac{\alpha}{\beta}\mu_\beta(P^{\perp}). \]
Thus $P$ is an order $\alpha$ digital $(t_{\alpha},m,s)$-net over $\FF_b$ with
\begin{align*}
 t_{\alpha} & =\alpha m - \mu_\alpha(P^{\perp}) +1 \leq \alpha m - \frac{\alpha}{\beta}\mu_\beta(P^{\perp}) +1\\
 & = \frac{\alpha}{\beta}\left( \beta m- \mu_\beta(P^{\perp}) +1\right) + \frac{\beta-\alpha}{\beta} = \frac{\alpha}{\beta}t_{\beta}+ \frac{\beta-\alpha}{\beta}.
\end{align*}
Given that $t_{\alpha}$ is a non-negative integer and that the fraction $(\beta-\alpha)/\beta$ is less than 1, the $t_{\alpha}$-value can be bounded above by $\lceil t_{\beta} \alpha/\beta \rceil$. The result for an order $\beta$ digital $(t_{\beta},s)$-sequence follows immediately.
\end{proof}
\noindent 
Importantly this result implies that any higher order digital nets and sequences are also digital $(t,m,s)$-nets and digital $(t,s)$-sequences, respectively.

We end this subsection by providing two useful results in analyzing the integration error of QMC rules using higher order digital nets. The first lemma is a higher order version of Lemma~\ref{lem:digital-net-card}.
\begin{lemma}\label{lem:HO_digital-net-card} For $\alpha\geq 2$, let $P$ be an order $\alpha$ digital $(t_{\alpha},m,s)$-net over $\FF_b$. The following holds true:
\begin{enumerate}
\item For $z\in \NN_0$,
\begin{align*}
 \left| \left\{ \bsk\in P^{\perp}\setminus \{\bszero\}\mid \mu_\alpha(\bsk)=z \right\} \right| \leq \begin{cases} 0 & \text{if $z<\mu_\alpha(P^{\perp})$,} \\ (b-1)^{s\alpha }b^{(z-\mu_\alpha(P^{\perp}))/\alpha}(z+2)^{s\alpha -1} & \text{otherwise.} \end{cases} 
\end{align*}
\item For any real $\lambda>1/\alpha$,
\[ \sum_{\bsk\in P^{\perp}\setminus \{\bszero\}} b^{-\lambda \mu_{\alpha}(\bsk)} \leq 2^{s\alpha-1}(b-1)^{s\alpha }\frac{(\mu_\alpha(P^{\perp}))^{s\alpha -1}}{b^{\lambda \mu_\alpha(P^{\perp})}}\sum_{z=0}^{\infty}b^{(1/\alpha-\lambda) z}(z+2)^{s\alpha -1} .\]
\end{enumerate}
\end{lemma}
\begin{proof}
Similarly to the proof of Lemma~\ref{lem:digital-net-card}, we put $A_{\alpha,z}=\left| \left\{ \bsk\in P^{\perp}\setminus \{\bszero\}\mid \mu_\alpha(\bsk)=z \right\} \right|$. For $1\leq j\leq s$, denote the $b$-adic expansion of $k_j\in \NN_0$ by
\[ k_j=\kappa_{1,j} b^{c_{1,j}-1}+\kappa_{2,j} b^{c_{2,j}-1}+\cdots + \kappa_{v_j,j} b^{c_{v_j,j}-1} \]
with $\kappa_{1,j},\ldots,\kappa_{v_j,j}\in \FF_b\setminus \{0\}$ and $c_{1,j}>\cdots >c_{v_j,j}>0$, and write $c_{v_j+1,j}=c_{v_j+2,j}=\cdots =0$. If $k_j=0$, let $c_{1,j}=c_{2,j}=\cdots=0$. Then we have
\begin{align*}
A_{\alpha,z} = \sum_{\substack{z_{1,j}\geq \cdots \geq z_{\alpha,j}\in \NN_0, \forall j=1,\ldots, s\\ z_{1,1}+\cdots +z_{\alpha,1}+\cdots+z_{1,s}+\cdots+z_{\alpha,s}=z}}\left| \left\{ \bsk\in P^{\perp}\setminus \{\bszero\}\mid c_{i,j}=z_{i,j}, 1\leq i\leq \alpha, 1\leq j\leq s \right\} \right| .
\end{align*}
It can be inferred from \cite[Proof of Lemma~15.20]{DPbook} that the summand is bounded above by
\begin{align*}
& \left| \left\{ \bsk\in P^{\perp}\setminus \{\bszero\}\mid c_{i,j}=z_{i,j}, 1\leq i\leq \alpha, 1\leq j\leq s \right\} \right| \\
& \leq \begin{cases} 0 & \text{if $z<\mu_\alpha(P^{\perp})$,} \\ (b-1)^{s\alpha} & \text{if $z\geq \mu_\alpha(P^{\perp})$ and $z_{\alpha,1}+\cdots+z_{\alpha,s}<\mu_\alpha(P^{\perp})/\alpha$,}  \\(b-1)^{s\alpha}b^{z_{\alpha,1}+\cdots+z_{\alpha,s}-\mu_\alpha(P^{\perp})/\alpha} & \text{if $z\geq \mu_\alpha(P^{\perp})$ and $z_{\alpha,1}+\cdots+z_{\alpha,s}\geq \mu_\alpha(P^{\perp})/\alpha$.} \end{cases}
\end{align*}
Thus, we have $A_{\alpha,z}=0$ if $z<\mu_\alpha(P^{\perp})$, since each summand is 0. For $z\geq \mu_\alpha(P^{\perp})$, this bound gives
\begin{align*}
A_{\alpha,z} & \leq  (b-1)^{s\alpha} \sum_{\substack{z_{1,j}\geq \cdots \geq z_{\alpha,j}\in \NN_0, \forall j=1,\ldots, s\\ z_{1,1}+\cdots +z_{\alpha,1}+\cdots+z_{1,s}+\cdots+z_{\alpha,s}=z}} \max(1,b^{z_{\alpha,1}+\cdots+z_{\alpha,s}-\mu_\alpha(P^{\perp})/\alpha}) \\
& \leq (b-1)^{s\alpha} \sum_{i=0}^{\lfloor z/\alpha \rfloor}\binom{i+s-1}{s-1}\binom{z-i+s(\alpha-1)-1}{s(\alpha-1)-1} \max(1, b^{i-\mu_\alpha(P^{\perp})/\alpha})\\
& \leq (b-1)^{s\alpha} \sum_{i=0}^{\lfloor z/\alpha\rfloor}(i+1)^{s-1}(z-i+1)^{s(\alpha-1)-1}\max(1, b^{i-\mu_\alpha(P^{\perp})/\alpha}) \\
& \leq (b-1)^{s\alpha} (z+2)^{s\alpha-2} \sum_{i=0}^{\lfloor z/\alpha \rfloor} \max(1, b^{i-\mu_\alpha(P^{\perp})/\alpha}) \\
& \leq (b-1)^{s\alpha} (z+2)^{s\alpha-2}(z/\alpha+1)b^{z/\alpha-\mu_\alpha(P^{\perp})/\alpha},
\end{align*}
which proves the first assertion of the lemma.

Using the result of the first assertion and then applying the change of variables $z \mapsto z+\mu_{\alpha}(P)$, we have
\begin{align*}
\sum_{\bsk \in P^\perp \setminus \{\bszero\}} b^{-\lambda \mu_\alpha(\bsk)} & = \sum_{z=\mu_\alpha(P^{\perp})}^{\infty}b^{-\lambda z}A_{\alpha,z} \\
& \leq (b-1)^{s\alpha }\sum_{z=\mu_\alpha(P^{\perp})}^{\infty}b^{-\lambda z+(z-\mu_\alpha(P^{\perp}))/\alpha}(z+2)^{s\alpha -1} \\
& \leq (b-1)^{s\alpha }b^{-\lambda \mu_\alpha(P^{\perp})}\sum_{z=0}^{\infty}b^{(1/\alpha-\lambda) z}(z+\mu_\alpha(P^{\perp})+2)^{s\alpha -1} \\
& \leq 2^{s\alpha-1}(b-1)^{s\alpha }(\mu_\alpha(P^{\perp}))^{s\alpha -1}b^{-\lambda \mu_\alpha(P^{\perp})}\sum_{z=0}^{\infty}b^{(1/\alpha-\lambda) z}(z+2)^{s\alpha -1}, 
\end{align*}
where the last sum over $z$ is finite since $\lambda>1/\alpha$. Hence we complete the proof.
\end{proof}

Before stating the second useful lemma, we need to recall the notion of ``type $(p,q)$'' introduced in \cite{GSY3}.
\begin{definition}\label{def:type}
For $k,l\in \NN_0$, we denote the $b$-adic expansions of $k$ and $l$ by
$$k=\sum_{i=1}^{v}\kappa_ib^{c_i-1}\quad \text{and}\quad l=\sum_{i=1}^{w}\lambda_ib^{d_i-1},$$
respectively, where $\kappa_1,\ldots,\kappa_v,\lambda_1,\ldots,\lambda_w\in \{1,\ldots,b-1\}$, $c_1>c_2>\cdots >c_v>0$ and $d_1>d_2>\cdots >d_w>0$.
For $k=0$ ($l=0$, resp.), we assume that $v=0$ and $\kappa_0b^{c_0-1}=0$ ($w=0$ and $\lambda_0b^{d_0-1}=0$, resp.).
For $p,q\in \NN_0$, we write
$$k^{(p)}=\sum_{i=p+1}^{v}\kappa_ib^{c_i-1}\quad \text{and}\quad l^{(q)}=\sum_{i=q+1}^{w}\lambda_ib^{d_i-1},$$
where the empty sum equals 0.
Then we say that $(k,l)$ is of type $(p,q)$ if $k^{(p)}=l^{(q)}$ and $\kappa_{p}b^{c_p-1}\neq \lambda_{q}b^{d_q-1}$, where we set $\kappa_{0}b^{c_0-1}=\lambda_{0}b^{d_0-1}=0$, except the case $k=l$ where we say that $(k,l)$ is of type $(0,0)$.
\end{definition}
\noindent In what follows, we write $(k,l)\in T_{\geq \alpha}$ if $(k,l)$ is of type $(p,q)$ with $p+q\geq \alpha$. In case of vectors in $\NN_0^s$ we write $(\bsk,\bsl)\in T_{\geq \alpha}$ if there exists at least one index $1\leq j\leq s$ such that $(k_j,l_j)\in T_{\geq \alpha}$.

Now the following result, which can be regarded as a generalization of the result shown in \cite[Lemma~3.7]{DP14}, is proven in \cite[Lemma~8]{GSY3}. Here we state the result in a slightly more general form.
\begin{lemma}\label{lem:HO_digital-net-card2} For $\alpha\in \NN$, let $P$ be an order $\alpha$ digital $(t_{\alpha},m,s)$-net over $\FF_b$. For $z\in \NN_0$, there exists a constant $B_{\alpha,b,s,t_{\alpha}}>0$ such that the following holds:
\begin{align*}
& \left| \left\{ (\bsk,\bsl)\in (P^{\perp}\setminus \{\bszero\})^2 \,|\, \mu_1(\bsk)+\mu_1(\bsl)=z, (\bsk,\bsl)\notin T_{\geq \alpha} \right\} \right| \\
& \leq \begin{cases} 0 & \text{if $z<2\mu_1(P^{\perp})$,} \\ B_{\alpha,b,s,t_{\alpha}}(z-2\mu_1(P^{\perp}))^{s(\alpha-1)+1}z^{s-1}b^{(z-2\mu_1(P^{\perp}))/2} & \text{otherwise.} \end{cases} 
\end{align*}
\end{lemma}

\subsection{Digit interlacing construction}\label{subsec:digit-interlacing}
Here we give an explicit construction of higher order digital nets and sequences based on the \emph{digit interlacing function} due to Dick \cite{Dic07,Dic08}:
\begin{definition}\label{def:interlacing}
Let $\alpha\in \NN$ and let $\bsx=(x_1,\ldots,x_{\alpha})\in [0,1]^{\alpha}$.
For $1\leq j\leq \alpha$, we denote the $b$-adic expansion of $x_j$ by $x_j=\sum_{i=1}^{\infty}\xi_{i,j}b^{-i}$. The digit interlacing function (of factor $\alpha)$ $\Dcal_{\alpha}: [0,1]^{\alpha}\to [0,1]$ is defined by
\[  \Dcal_{\alpha}(x_1,\ldots,x_{\alpha}) := \sum_{i=1}^{\infty}\sum_{j=1}^{\alpha}\frac{\xi_{i,j}}{b^{\alpha(i-1) +j}} .\]
In case of vectors in $[0,1]^{\alpha s}$, we apply $\Dcal_{\alpha}$ to every non-overlapping consecutive $\alpha$ components, i.e.,
\[  \Dcal_{\alpha}(x_1,\ldots,x_{\alpha s}) := (\Dcal_{\alpha}(x_1,\ldots,x_{\alpha}),\ldots,\Dcal_{\alpha}(x_{\alpha(s-1)+1},\ldots,x_{\alpha s}))\in [0,1]^s .\]
\end{definition}
\begin{lemma}\label{lem:ho_construction} The following holds true:
\begin{enumerate}
\item Let $P$ be a digital $(t,m,\alpha s)$-net over $\FF_b$. The set $\Dcal_{\alpha}(P) = \{\Dcal_{\alpha}(\bsx)\mid \bsx\in P\}$ is an order $\alpha$ digital $(t_{\alpha},m,s)$-net over $\FF_b$ with
\[ t_{\alpha} \leq \alpha \min\left( m,t+\left\lfloor \frac{s(\alpha-1)}{2}\right\rfloor\right). \]
\item Let $\Scal$ be a digital $(t,\alpha s)$-sequence over $\FF_b$. The sequence $\Dcal_{\alpha}(\Scal) = \{\Dcal_{\alpha}(\bsx)\mid \bsx\in \Scal\}$ is an order $\alpha$ digital $(t_{\alpha},s)$-sequence over $\FF_b$ with
\[ t_{\alpha} \leq \alpha t+ \frac{s\alpha (\alpha-1)}{2}. \]
\end{enumerate}
\end{lemma}
\noindent 
Since there are many explicit constructions of digital $(t,m,s)$-nets and $(t,s)$-sequences with small $t$-value for an arbitrarily dimension $s$, as described in Subsection~\ref{subsec:standard_constr}, the above lemma from \cite[Theorems~4.11 and 4.12]{Dic08} directly implies that higher order digital nets and sequences can be explicitly constructed.

\begin{remark}
Let $P$ be a digital $(t,m,\alpha s)$-net over $\FF_b$ with generating matrices $C_1,\ldots,C_{\alpha s}\in \FF_b^{m\times m}$. Let $\bsc_i^{(j)}$ denote the $i$-th row vector of $C_j$. For each $1\leq j\leq s$, construct the matrix $D_j\in \FF_b^{\alpha m\times m}$, whose $i$-th row vector is denoted by $\bsd_i^{(j)}$, from the matrices $C_{\alpha(j-1)+1},\ldots, C_{\alpha j}$ as
\[ \bsd_{\alpha(h-1)+i}^{(j)}= \bsc_h^{(\alpha (j-1)+i)},\quad \text{for $1\leq h\leq m$ and $1\leq i\leq \alpha$.}\]
Then the set $\Dcal_{\alpha}(P)$ is a digital net over $\FF_b$ with generating matrices $D_1,\ldots,D_s$. 

Similarly, the sequence $\Dcal_{\alpha}(\Scal)$ can be identified with a digital sequence over $\FF_b$ with generating matrices $D_1,\ldots,D_s\in \FF_b^{\NN\times \NN}$ which are constructed from the generating matrices $C_1,\ldots,C_{\alpha s}\in \FF_b^{\NN\times \NN}$ of $\Scal$.
\end{remark}

To construct an interlaced finite point set $\Dcal_{\alpha}(P)$, one can use polynomial lattice point sets in dimension $\alpha s$ instead of digital $(t,m,\alpha s)$-nets. The resulting point set $\Dcal_{\alpha}(P)$ is called an \emph{interlaced polynomial lattice point set}, and has been often used in applications of HoQMC methods, see \cite{DKLNS14,DKLS16,DLS16,DGLS17,GHS18a,GHS18b,GP18,DGLS19}. Here we need to find good generating vectors $\bsq=(q_1,\ldots,q_{\alpha s})$, but the digit interlacing composition makes it nontrivial whether each component $q_j$ can be searched for one-by-one or consecutive $\alpha$ components $q_{\alpha(j-1)+1},\ldots,q_{\alpha j}$ should be searched for simultaneously. The papers \cite{GD,G1} originally gave a justification for employing the former approach, i.e., component-by-component construction.

\section{Walsh functions}\label{sec:Walsh}

\subsection{Definitions}
The Walsh functions were originally introduced by Walsh \cite{Wal23} and have been studied thereafter, for instance, in \cite{Fin49,Chr55}. In what follows, let $\omega_b$ denote the primitive $b$-th root of unity $\exp(2\pi \sqrt{-1}/b)$. The one-dimensional Walsh functions are defined as follows.
\begin{definition}
Let $k\in \NN_0$ with $b$-adic expansion $k=\sum_{i=0}^{\infty}\kappa_ib^i$, where all but a finite number of $\kappa_i$ are 0. The $k$-th $b$-adic Walsh function $\wal_k\colon [0,1]\to \{1,\omega_b,\ldots,\omega_b^{b-1}\}$ is defined by
\[ \wal_k (x) := \omega_b^{\kappa_0\xi_1+\kappa_1\xi_2+\cdots},\]
where the unique $b$-adic expansion of $x\in [0,1]$ is denoted by $x=\sum_{i=1}^{\infty}\xi_i b^{-i}$.
\end{definition}
\noindent It is clear from the definition that every Walsh function is piecewise constant since it depends only on some finite number of the digits $\xi_i$. Multi-dimensional Walsh functions are given by generalizing the one-dimensional Walsh functions:
\begin{definition}
Let $s\geq 1$. For $\bsk=(k_1,\ldots,k_s)\in \NN_0^s$, the $\bsk$-th $b$-adic Walsh function $\wal_{\bsk}\colon [0,1]^s\to \{1,\omega_b,\ldots,\omega_b^{b-1}\}$ is defined by
\[ \wal_{\bsk} (\bsx) := \prod_{j=1}^{s}\wal_{k_j}(x_j). \]
\end{definition}

Several important properties of the Walsh functions are listed below, see \cite[Appendix~A.2]{DPbook} for the proof.
\begin{lemma}\label{lem:walsh} The following holds true:
\begin{enumerate}
\item For $\bsk,\bsl\in \NN_0^s$ and $\bsx,\bsy\in [0,1]^s$,
\[ \wal_{\bsk}(\bsx)\wal_{\bsl}(\bsx) = \wal_{\bsk\oplus \bsl}(\bsx) \quad \text{and}\quad \wal_{\bsk}(\bsx)\wal_{\bsk}(\bsy) = \wal_{\bsk}(\bsx\oplus \bsy). \]
\item For $\bsk\in \NN_0^s$, 
\[ \int_{[0,1]^s}\wal_{\bsk}(\bsx)\rd \bsx = \begin{cases} 1 & \text{if $\bsk=\bszero$,} \\ 0 & \text{otherwise.}\end{cases} \]
\item For $\bsk,\bsl\in \NN_0^s$, 
\[ \int_{[0,1]^s}\wal_{\bsk}(\bsx)\overline{\wal_{\bsl}(\bsx)}\rd \bsx = \begin{cases} 1 & \text{if $\bsk=\bsl$,} \\ 0 & \text{otherwise.}\end{cases} \]
\item For any $s\in \NN$, the Walsh system $\{\wal_{\bsk} \mid \bsk\in \NN_0^s\}$ is a complete orthonormal system in $L_2([0,1]^s)$
\end{enumerate}
\end{lemma}
\noindent It follows from the fourth assertion of Lemma~\ref{lem:walsh} that we can define the Walsh series of $f\in L_2([0,1]^s)$:
\[ \sum_{\bsk\in \NN_0^s}\hat{f}(\bsk) \wal_{\bsk}(\bsx), \]
where $\hat{f}(\bsk)$ denotes the $\bsk$-th Walsh coefficient of $f$ defined by
\[ \hat{f}(\bsk) := \int_{[0,1]^s}f(\bsx)\overline{\wal_{\bsk}(\bsx)}\rd \bsx. \]
For any continuous function $f$ which satisfies $\sum_{\bsk\in \NN_0^s}|\hat{f}(\bsk)|<\infty$, the above Walsh series of $f$ equals pointwise to $f$ itself, see \cite[Theorem~A.20]{DPbook}.

\subsection{Connection to digital nets}
It follows from the first assertion of Lemma~\ref{lem:walsh} that the Walsh functions hold the following important character property. The proof can be found, for instance, in \cite[Lemma~4.75]{DPbook}.
\begin{lemma}\label{lem:dual-walsh}
Let $P\subset [0,1]^s$ be a digital net over $\FF_b$. For $\bsk\in \NN_0^s$ we have
\[ \sum_{\bsx\in P}\wal_{\bsk}(\bsx) = \begin{cases} |P| & \text{if $\bsk\in P^{\perp}$,} \\ 0 & \text{otherwise.}\end{cases} \]
\end{lemma}
\noindent In what follows, by using this lemma, we show how the Walsh functions play a crucial role in analyzing the QMC integration error. 

As preparation, let us consider a reproducing kernel Hilbert space $H$ equipped with reproducing kernel $K\colon [0,1]^s\times [0,1]^s\to \RR$ and inner product $\langle \cdot, \cdot \rangle_{K}$. The norm of $f\in H$ is simply given by $\|f\|_K = \sqrt{\langle f, f \rangle_{K}}$. The worst-case error of QMC integration using a point set $P$ is defined by
\[ e^{\wor}(H, P) := \sup_{\substack{f\in H\\ \| f\|_K\leq 1}}| I(f)- I(f;P)|, \]
while the initial error is defined as reference by
\[ e^{\wor}(H, 0) := \sup_{\substack{f\in H\\ \| f\|_K\leq 1}}| I(f)|. \]
Both of the initial error and the worst-case error have explicit formulas relying only on $K$ and $P$ as follows, see \cite[Chapter~2.3.3]{DPbook} for the proof.
\begin{proposition}
For a reproducing kernel Hilbert space $H$ whose reproducing kernel satisfies $\int_{[0,1]^s}\sqrt{K(\bsx,\bsx)}\rd \bsx < \infty$, the squared initial error is given by
\[ (e^{\wor}(H, 0))^2 = \int_{[0,1]^{2s}}K(\bsx,\bsy)\rd \bsx \rd \bsy. \]
The squared worst-case error of QMC integration using a point set $P$ is given by
\begin{align}\label{eq:worst-case_error}
 (e^{\wor}(H, P))^2 & = \int_{[0,1]^{2s}}K(\bsx,\bsy)\rd \bsx \rd \bsy - \frac{2}{|P|}\sum_{\bsx\in P} \int_{[0,1]^s}K(\bsx,\bsy)\rd \bsy \notag \\
 & \quad + \frac{1}{|P|^2}\sum_{\bsx,\bsy\in P}K(\bsx,\bsy). 
\end{align}
\end{proposition}

Now let us consider the Walsh series of a reproducing kernel $K$:
\[ \sum_{\bsk,\bsl\in \NN_0^s}\hat{K}(\bsk,\bsl)\wal_{\bsk}(\bsx)\overline{\wal_{\bsl}(\bsy)}, \]
where the $(\bsk,\bsl)$-th Walsh coefficient $\hat{K}(\bsk,\bsl)$ is defined by
\[ \hat{K}(\bsk,\bsl) = \int_{[0,1]^{2s}}K(\bsx,\bsy)\overline{\wal_{\bsk}(\bsx)}\wal_{\bsl}(\bsy)\rd \bsx \rd \bsy. \]
Again the pointwise equality holds between the above Walsh series and $K$ itself if $K$ is continuous and $\sum_{\bsk,\bsl\in \NN_0^s}|\hat{K}(\bsk,\bsl)|<\infty$. 

The following proposition provides a simple expression of the squared worst-case error when the point set is a \emph{digitally shifted digital net}
\[ P\oplus \bsdelta = \{\bsx\oplus \bsdelta \mid \bsx\in P\},\]
where $P$ is a digital net over $\FF_b$ and $\bsdelta\in [0,1]^s$. As far as the authors know, this result is not available in the literature in this full generality, so that we provide a proof for the sake of completeness.
\begin{proposition}\label{prop:worst-case-error_dual}
Let $P$ be a digital net over $\FF_b$ and $\bsdelta\in [0,1]^s$. For a reproducing kernel Hilbert space $H$ whose reproducing kernel $K$ is continuous and satisfies $\int_{[0,1]^s}\sqrt{K(\bsx,\bsx)}\rd \bsx < \infty$ and $\sum_{\bsk,\bsl\in \NN_0^s}|\hat{K}(\bsk,\bsl)|<\infty$, we have
\[ (e^{\wor}(H, P\oplus \bsdelta))^2 = \sum_{\bsk,\bsl\in P^{\perp}\setminus \{\bszero\}}\hat{K}(\bsk,\bsl)\wal_{\bsk}(\bsdelta)\overline{\wal_{\bsl}(\bsdelta)}. \]
\end{proposition}
\begin{proof}
It is trivial from the definition of the Walsh functions that the first term on the right-hand side of \eqref{eq:worst-case_error} is equal to $\hat{K}(\bszero,\bszero)$. For the second term on the right-hand side of \eqref{eq:worst-case_error}, by using the symmetry of $K$, the first and second assertions of Lemma~\ref{lem:walsh} and Lemma~\ref{lem:dual-walsh}, we have
\begin{align*}
& \frac{2}{|P|}\sum_{\bsx\in P} \int_{[0,1]^s}K(\bsx\oplus \bsdelta,\bsy)\rd \bsy \\
& = \frac{1}{|P|}\sum_{\bsx\in P} \int_{[0,1]^s}K(\bsx\oplus \bsdelta,\bsy)\rd \bsy +\frac{1}{|P|}\sum_{\bsx\in P} \int_{[0,1]^s}K(\bsy,\bsx\oplus \bsdelta)\rd \bsy \\
& = \frac{1}{|P|}\sum_{\bsx\in P} \sum_{\bsk,\bsl\in \NN_0^s}\hat{K}(\bsk,\bsl)\wal_{\bsk}(\bsx\oplus \bsdelta)\int_{[0,1]^s}\overline{\wal_{\bsl}(\bsy)} \rd \bsy \\
& \quad + \frac{1}{|P|}\sum_{\bsx\in P} \sum_{\bsk,\bsl\in \NN_0^s}\hat{K}(\bsk,\bsl)\overline{\wal_{\bsl}(\bsx\oplus \bsdelta)}\int_{[0,1]^s}\wal_{\bsk}(\bsy) \rd \bsy \\
& = \sum_{\bsk\in \NN_0^s}\hat{K}(\bsk,\bszero)\wal_{\bsk}(\bsdelta)\frac{1}{|P|}\sum_{\bsx\in P} \wal_{\bsk}(\bsx) \\
& \quad + \sum_{\bsl\in \NN_0^s}\hat{K}(\bszero,\bsl)\overline{\wal_{\bsl}(\bsdelta)}\frac{1}{|P|}\sum_{\bsx\in P} \overline{\wal_{\bsl}(\bsx)} \\
& = \sum_{\bsk\in P^{\perp}}\hat{K}(\bsk,\bszero)\wal_{\bsk}(\bsdelta) + \sum_{\bsl\in P^{\perp}}\hat{K}(\bszero,\bsl)\overline{\wal_{\bsl}(\bsdelta)}.
\end{align*}
Finally, for the third term on the right-hand side of \eqref{eq:worst-case_error}, by using the first assertion of Lemma~\ref{lem:walsh} and Lemma~\ref{lem:dual-walsh}, we have
\begin{align*}
& \frac{1}{|P|^2}\sum_{\bsx,\bsy\in P}K(\bsx\oplus \bsdelta,\bsy\oplus \bsdelta) \\
& = \frac{1}{|P|^2}\sum_{\bsx,\bsy\in P}\sum_{\bsk,\bsl\in \NN_0^s}\hat{K}(\bsk,\bsl)\wal_{\bsk}(\bsx\oplus \bsdelta)\overline{\wal_{\bsl}(\bsy\oplus \bsdelta)} \\
& = \sum_{\bsk,\bsl\in \NN_0^s}\hat{K}(\bsk,\bsl)\wal_{\bsk}(\bsdelta)\overline{\wal_{\bsl}(\bsdelta)}\frac{1}{|P|}\sum_{\bsx\in P}\wal_{\bsk}(\bsx)\frac{1}{|P|}\sum_{\bsy\in P}\overline{\wal_{\bsl}(\bsy)} \\
& = \sum_{\bsk,\bsl\in P^{\perp}}\hat{K}(\bsk,\bsl)\wal_{\bsk}(\bsdelta)\overline{\wal_{\bsl}(\bsdelta)}.
\end{align*}
Altogether we obtain
\begin{align*}
(e^{\wor}(H, P\oplus \bsdelta))^2 & = \hat{K}(\bszero,\bszero)-\sum_{\bsk\in P^{\perp}}\hat{K}(\bsk,\bszero)\wal_{\bsk}(\bsdelta) - \sum_{\bsl\in P^{\perp}}\hat{K}(\bszero,\bsl)\overline{\wal_{\bsl}(\bsdelta)} \\
& \quad + \sum_{\bsk,\bsl\in P^{\perp}}\hat{K}(\bsk,\bsl)\wal_{\bsk}(\bsdelta)\overline{\wal_{\bsl}(\bsdelta)}\\
& = \sum_{\bsk,\bsl\in P^{\perp}\setminus \{\bszero\}}\hat{K}(\bsk,\bsl)\wal_{\bsk}(\bsdelta)\overline{\wal_{\bsl}(\bsdelta)},
\end{align*}
which completes the proof.
\end{proof}

Using Proposition~\ref{prop:worst-case-error_dual} we obtain the following result.
\begin{corollary}\label{cor:worst-case-error_dual}
Let $P$ be a digital net over $\FF_b$ and $H$ be a reproducing kernel Hilbert space whose reproducing kernel $K$ is continuous and satisfies $\int_{[0,1]^s}\sqrt{K(\bsx,\bsx)}\rd \bsx < \infty$ and $\sum_{\bsk,\bsl\in \NN_0^s}|\hat{K}(\bsk,\bsl)|<\infty$. Then we have
\[ (e^{\wor}(H, P))^2 = \sum_{\bsk,\bsl \in P^{\perp}\setminus \{\bszero\}}\hat{K}(\bsk,\bsl), \]
and
\[ \int_{[0,1]^s}(e^{\wor}(H, P\oplus \bsdelta))^2\rd \bsdelta = \sum_{\bsk\in P^{\perp}\setminus \{\bszero\}}\hat{K}(\bsk,\bsk). \]
\end{corollary}
\begin{proof}
The first assertion follows immediately from Proposition~\ref{prop:worst-case-error_dual} by considering the case $\bsdelta=\bszero$. For the second assertion, it follows from Proposition~\ref{prop:worst-case-error_dual} and the third assertion of Lemma~\ref{lem:walsh} that
\begin{align*}
\int_{[0,1]^s}(e^{\wor}(H, P\oplus \bsdelta))^2\rd \bsdelta & = \sum_{\bsk,\bsl\in P^{\perp}\setminus \{\bszero\}}\hat{K}(\bsk,\bsl)\int_{[0,1]^s}\wal_{\bsk}(\bsdelta)\overline{\wal_{\bsl}(\bsdelta)}\rd \bsdelta \\
& =  \sum_{\bsk\in P^{\perp}\setminus \{\bszero\}}\hat{K}(\bsk,\bsk).
\end{align*}
Thus we are done.
\end{proof}

We would like to emphasize that the worst-case error of QMC rule using a digital net is given by the double sum of the Walsh coefficients of reproducing kernel, whereas the shift-averaged worst-case error is simply given by the single sum of their diagonal elements. This way applying a random digital shift has an effect on vanishing all non-diagonal terms, which sometimes makes an error analysis much easier as we shall see in the subsequent sections.

\section{Discrepancy}\label{sec:discrepancy}

\subsection{Definitions}
As represented by the Koksma-Hlawka inequality \eqref{eq:KH-inequ}, the star-discrepancy, or more generally speaking, the $L_p$-discrepancy is an extremely important quantitative measure of how uniformly a point set is distributed over $[0,1]^s$.
\begin{definition}\label{def:discrepancy}
Let $P$ be a point set in $[0,1]^s$. For $\bsy=(y_1,\ldots,y_s)\in [0,1]^s$, we write $[\bszero,\bsy)=[0,y_1)\times \cdots \times [0,y_s)$ and define the so-called \emph{local discrepancy function}
\[ \Delta_P(\bsy) := \frac{1}{|P|}\sum_{\bsx\in P}\bsone_{[\bszero,\bsy)}(\bsx)-\prod_{j=1}^{s}y_j, \]
where $\bsone_{[\bszero,\bsy)}$ denotes the indicator function which is equal to 1 if $\bsx\in [\bszero,\bsy)$, and 0 otherwise. For $1\leq p\leq \infty$, the $L_p$-discrepancy of $P$ is defined as the $L_p$-norm of $\Delta_P(\bsy)$, i.e.,
\[ L_p(P) := \left( \int_{[0,1]^s}|\Delta_P(\bsy)|^p\rd \bsy\right)^{1/p}, \]
with the obvious modification for $p=\infty$.
\end{definition}
\noindent We speak of the star-discrepancy if $p=\infty$ and the different notions such as $D^{*}(P)$ have been often used in the literature.

\subsection{$L_2$-discrepancy and worst-case error}
In this section, we are particularly concerned with the $L_2$-discrepancy. Here we follow the exposition of \cite[Chapter~2.4]{DPbook} to give a connection between the $L_2$-discrepancy and the worst-case error for some reproducing kernel Hilbert space. 

Let $H^{\downarrow}$ denote the reproducing kernel Hilbert space whose reproducing kernel is given by
\[ K^{\downarrow}(\bsx,\bsy) = \prod_{j=1}^{s}\min(1-x_j,1-y_j). \]
It is known from \cite[Section~8]{Aro50} that the function space $H^{\downarrow}$ is the $s$-fold tensor product of the univariate reproducing kernel Hilbert space with reproducing kernel $K^{\downarrow}(x,y) = \min(1-x,1-y)$. Moreover, $H^{\downarrow}$ contains all absolutely continuous functions which satisfy
\[ \frac{\partial^{|u|}f}{\partial \bsx_u}(\bsx_u,\bsone)=0\quad \text{for all $\emptyset = u\subsetneq \{1,\ldots,s\}$,}\]
where we write $\bsx_u=(x_j)_{j\in u}$ and $(\bsx_u,\bsone)$ is the vector whose $j$-th component is equal to $x_j$ if $j\in u$, and 1 otherwise. Note that the symbol $\downarrow$ is used to represent this ``anchor'' property of the space. The inner product is given by
\[ \langle f,g\rangle_{K^{\downarrow}} = \int_{[0,1]^s}\frac{\partial^{s}f}{\partial \bsx}(\bsx)\frac{\partial^{s}g}{\partial \bsx}(\bsx)\rd \bsx, \]
for $f,g\in H^{\downarrow}$. Then the following result is known.
\begin{lemma}For any point set $P\subset [0,1]^s$, we have
\[ L_2(P) = e^{\wor}(H^{\downarrow}, P). \]
\end{lemma}

Therefore, when $P$ is a digital net over $\FF_b$, combining this identity with Corollary~\ref{cor:worst-case-error_dual} gives an expression of the squared $L_2$-discrepancy:
\[ (L_2(P))^2 = \sum_{\bsk,\bsl \in P^{\perp}\setminus \{\bszero\}}\widehat{K^{\downarrow}}(\bsk,\bsl), \]
and also an expression of the shift-averaged squared $L_2$-discrepancy:
\[ \int_{[0,1]^s}(L_2(P\oplus \bsdelta))^2\rd \bsdelta = \sum_{\bsk\in P^{\perp}\setminus \{\bszero\}}\widehat{K^{\downarrow}}(\bsk,\bsk). \]

In \cite[Lemma~2.2]{DP14}, Dick and Pillichshammer conducted an exact evaluation of the Walsh coefficients $\widehat{K^{\downarrow}}(\bsk,\bsl)$ for all $\bsk,\bsl\in \NN_0^s$ when $b=2$. Here we simplify their result in a way that will be sufficient for readers to understand the proof of optimal order $L_2$-discrepancy bounds.
\begin{lemma}\label{lem:walsh_L2disc}
For $\bsk,\bsl\in \NN_0^s$, we have 
\begin{enumerate}
\item $|\widehat{K^{\downarrow}}(\bsk,\bsl)| \leq 3^{-s}\cdot 2^{-\mu_1(\bsk)-\mu_1(\bsl)}$.
\item $\widehat{K^{\downarrow}}(\bsk,\bsl)=0$ if $(\bsk,\bsl)\in T_{\geq 3}$.
\end{enumerate}
\end{lemma}
\noindent Here we recall that the notion $T_{\geq \alpha}$ was introduced in Subsection~\ref{subsec:HO_def}. The first assertion of the lemma shows the decay of the Walsh coefficients, while the second assertion shows the sparsity of the Walsh coefficients.

\subsection{Optimal order $L_2$-discrepancy bounds}\label{subsec:optimal-L_2-bounds}
Before stating the result by Dick and Pillichshammer \cite{DP14} on the $L_2$-discrepancy bound for higher order digital nets, we start from another result by the same authors shown in \cite{DP05} by using Lemma~\ref{lem:digital-net-card}. 

\begin{theorem}\label{thm:shifted-L_2-disc}
Let $P$ be a digital $(t,m,s)$-net over $\FF_2$. Then there exists a constant $D^{(1)}_{s,t}$ such that the following holds:
\[ \int_{[0,1]^s}(L_2(P\oplus \bsdelta))^2\rd \bsdelta \leq D^{(1)}_{s,t}\frac{m^{s-1}}{2^{2m}}. \]
\end{theorem}
\begin{proof}
Using Lemma~\ref{lem:walsh_L2disc} and the second assertion of Lemma~\ref{lem:digital-net-card} (with $\lambda=2$) we have
\begin{align*}
\int_{[0,1]^s}(L_2(P\oplus \bsdelta))^2\rd \bsdelta & = \sum_{\bsk\in P^{\perp}\setminus \{\bszero\}}\widehat{K^{\downarrow}}(\bsk,\bsk) \leq \frac{1}{3^s}\sum_{\bsk\in P^{\perp}\setminus \{\bszero\}}2^{-2\mu_1(\bsk)} \\
& \leq \frac{2^{s+1}}{3^s}\cdot \frac{(\mu_1(P^\perp))^{s-1}}{2^{2\mu_1(P^\perp)}} \sum_{z=1}^\infty 2^{-z} z^{s-1} .
\end{align*}
Since we have $\mu_1(P^{\perp})=m-t+1$, the result of the theorem follows.
\end{proof}
\noindent This theorem directly implies the existence of a digital shift $\bsdelta\in [0,1]^s$ such that the digitally shifted digital $(t,m,s)$-net satisfies the best possible order of $L_2$-discrepancy:
\[ L_2(P\oplus \bsdelta) \leq \sqrt{D^{(1)}_{s,t}}\frac{m^{(s-1)/2}}{2^{m}}=\sqrt{D^{(1)}_{s,t}}\frac{(\log_2 N)^{(s-1)/2}}{N} . \]
However, this is a probabilistic result since we do not know how to find such $\bsdelta$ explicitly.

We note that the optimal exponent $(s-1)$ of the numerator in Theorem~\ref{thm:shifted-L_2-disc} comes from the inequality
\[ \left| \left\{ \bsk\in P^{\perp}\setminus \{\bszero\}\mid \mu_1(\bsk)=z \right\} \right| \leq b^{z-\mu_1(P^{\perp})+1}(z+1)^{\underline{s-1}}  \]
given in the first assertion of Lemma~\ref{lem:digital-net-card}. As is clear from the proof, since the expression of the shift-averaged squared $L_2$-discrepancy has only the diagonal terms of the Walsh coefficients, there is no necessity of exploiting the sparsity of the Walsh coefficients. In order to obtain a deterministic counterpart of Theorem~\ref{thm:shifted-L_2-disc}, however, it seems insufficient to exploit the decay of the Walsh coefficients only, and one approach is to exploit both the decay and the sparsity of the Walsh coefficients simultaneously. To do this, we rely on Lemma~\ref{lem:HO_digital-net-card2}. The following theorem is from \cite[Theorem~4.1]{DP14}.

\begin{theorem}\label{thm:L_2-disc}
Let $P$ be an order 3 digital $(t_3,m,s)$-net over $\FF_2$. Then there exists a constant $D_{s,t_3}^{(2)}$ such that the following holds:
\[  (L_2(P))^2\leq D_{s,t_3}^{(2)}\frac{m^{s-1}}{2^{2m}}. \]
\end{theorem}
\begin{proof}
Using Lemmas~\ref{lem:walsh_L2disc} and \ref{lem:HO_digital-net-card2} (with $\alpha=3$) and then applying the change of variables $z\to z+2\mu_1(P^{\perp})$, we have
\begin{align*}
&\left(L_2(P)\right)^2 \\
& = \sum_{\bsk,\bsl\in P^{\perp}\setminus\{\bszero\}}\widehat{K^{\downarrow}}(\bsk,\bsl) \leq \frac{1}{3^s} \sum_{\substack{\bsk,\bsl\in P^{\perp}\setminus \{\bszero\}\\ (\bsk,\bsl)\notin T_{\geq 3}}}2^{-\mu_{1}(\bsk)-\mu_{1}(\bsl)} \\
& = \frac{1}{3^s} \sum_{z=2\mu_{1}(P^{\perp})}^{\infty}2^{-z} \left| \{ (\bsk,\bsl)\in (P^{\perp}\setminus \{\bszero\})^2\mid \mu_1(\bsk)+\mu_1(\bsl)=z, (\bsk,\bsl)\notin T_{\geq 3} \}\right| \\
& \leq \frac{B_{3,2,s,t_{3}}}{3^s}\sum_{z=2\mu_{1}(P^{\perp})}^{\infty}2^{-z} (z-2\mu_1(P^{\perp}))^{2s+1}z^{s-1}2^{(z-2\mu_1(P^{\perp}))/2} \\
& = \frac{B_{3,2,s,t_{3}}}{3^s}2^{-2\mu_1(P^{\perp})}\sum_{z=0}^{\infty}2^{-z/2} z^{2s+1}(z+2\mu_1(P^{\perp}))^{s-1}\\
& \leq \frac{B_{3,2,s,t_{3}}}{3^s}\cdot \frac{2^{2(s-1)}(\mu_1(P^{\perp}))^{s-1}}{2^{2\mu_1(P^{\perp})}}\sum_{z=0}^{\infty}2^{-z/2} z^{3s}.
\end{align*}
The last sum over $z$ is trivially finite. It follows from Lemma~\ref{lem:propagation} that an order 3 digital $(t_3,m,s)$-net over $\FF_2$ is also a digital $(t,m,s)$-net over $\FF_2$ with $t\leq \lceil t_3/3 \rceil$, i.e., $\mu_1(P^{\perp})\geq m-\lceil t_3/3 \rceil+1$. Thus, the result of the theorem follows.
\end{proof}

\begin{remark}
If we do not take the sparsity of the Walsh coefficients into account, we might proceed like
\begin{align*}
\left(L_2(P)\right)^2 & = \sum_{\bsk,\bsl\in P^{\perp}\setminus\{\bszero\}}\widehat{K^{\downarrow}}(\bsk,\bsl) \leq \frac{1}{3^s} \sum_{\bsk,\bsl\in P^{\perp}\setminus \{\bszero\}}2^{-\mu_{1}(\bsk)-\mu_{1}(\bsl)} \\
& = \frac{1}{3^s} \left( \sum_{\bsk\in P^{\perp}\setminus \{\bszero\}}2^{-\mu_{1}(\bsk)}\right)^2 \\
& = \frac{1}{3^s} \left( \sum_{z=\mu_1(P^{\perp})}^{\infty}2^{-z}\left| \left\{ \bsk\in P^{\perp}\setminus \{\bszero\}\mid \mu_1(\bsk)=z \right\} \right|\right)^2 \\
& \leq \frac{2^2}{3^s\cdot 2^{\mu_1(P^{\perp})}} \left( \sum_{z=\mu_1(P^{\perp})}^{\infty}(z+1)^{s-1}\right)^2,
\end{align*}
where we used Lemma~\ref{lem:digital-net-card} in the last inequality. Since the last sum over $z$ obviously diverges, this argument ends up with a trivial upper bound.
\end{remark}

We recall that Theorem~\ref{thm:L_2-disc} is for order 3 digital nets over $\FF_2$, a class of finite point sets. By considering order 5 digital sequences, as the other main result of the paper \cite{DP14}, Dick and Pillichshammer proved the optimal order $L_2$-discrepancy bound which holds uniformly for all $N$.

\section{Numerical integration}\label{sec:integration}

\subsection{Sobolev spaces}
Let us move onto multivariate numerical integration problem. The function space of our interest in this section is defined as follows. First let us consider the one-dimensional case. For $\alpha\in \NN$, the Sobolev space with smoothness $\alpha$ is
\begin{align*}
 H_{\alpha} & := \Big\{f \colon [0,1]\to \RR \mid \\
 & \qquad f^{(r)} \colon \text{absolutely continuous for $r=0,\ldots,\alpha-1$}, f^{(\alpha)}\in L^2([0,1])\Big\},
\end{align*}
where $f^{(r)}$ denotes the $r$-th derivative of $f$. According to \cite[Chapter~10.2]{Wbook}, the space $H_{\alpha}$ is a reproducing kernel Hilbert space with the reproducing kernel
\begin{align*}
 K_{\alpha}(x,y) = \sum_{r=0}^{\alpha}\frac{B_r(x)B_r(y)}{(r!)^2}+(-1)^{\alpha+1}\frac{B_{2\alpha}(|x-y|)}{(2\alpha)!} ,
\end{align*}
for $x,y\in [0,1]$, where $B_r$ denotes the Bernoulli polynomial of degree $r$, and with the inner product \begin{align*}
 \langle f, g \rangle_{K_\alpha} = \sum_{r=0}^{\alpha-1}\int_{0}^{1}f^{(r)}(x)\, \rd x \int_{0}^{1}g^{(r)}(x)\, \rd x + \int_{0}^{1}f^{(\alpha)}(x)g^{(\alpha)}(x)\, \rd x,
\end{align*}
for $f,g\in H_{\alpha}$.

For the $s$-dimensional case, we consider the $s$-fold tensor product space of the one-dimensional space introduced above. Thus the Sobolev space $H_{\alpha,s}$ which we consider is simply given by $H_{\alpha,s}=\bigotimes_{j=1}^{s} H_{\alpha}$. Again it is known from \cite[Section~8]{Aro50} that the reproducing kernel of the space $H_{\alpha,s}$ is the product of the reproducing kernels for the one-dimensional space $H_{\alpha}$.
Therefore, $H_{\alpha,s}$ is the reproducing kernel Hilbert space with the reproducing kernel
\begin{align*}
 K_{\alpha,s}(\bsx,\bsy) = \prod_{j=1}^{s}K_{\alpha}(x_j,y_j) ,
\end{align*}
and with the inner product
\begin{align*}
 \langle f, g \rangle_{K_{\alpha,s}} & = \sum_{u\subseteq \{1,\ldots,s\}}\sum_{\bsr_u\in \{0,\ldots,\alpha-1\}^{|u|}} \int_{[0,1]^{s-|u|}} \\
 & \qquad \left(\int_{[0,1]^{|u|}}f^{(\bsr_u,\bsalpha)}(\bsx)\, \rd \bsx_u\right) \left(\int_{[0,1]^{|u|}} g^{(\bsr_u,\bsalpha)}(\bsx) \, \rd \bsx_u\right) \, \rd \bsx_{\{1,\ldots,s\}\setminus u} ,
\end{align*}
for $f,g\in H_{\alpha,s}$. In the above, we used the following notation: For $u\subseteq \{1,\ldots,s\}$ and $\bsx\in [0,1]^s$, we write $\bsx_u=(x_j)_{j\in u}$. Moreover, for $\bsr_u=(r_j)_{j\in u}\in \{0,\ldots,\alpha-1\}^{|u|}$, $(\bsr_u,\bsalpha)$ denotes the $s$-dimensional vector whose $j$-th component equals $r_j$ if $j\in u$, and $\alpha$ otherwise. Note that an integral and sum over the empty set is defined to be the identity operator. Since the dimension $s$ is fixed, we shall simply write $H_{\alpha}$ and $K_{\alpha}$ instead of $H_{\alpha,s}$ and $K_{\alpha,s}$, respectively. 

If $P$ is a digital net over $\FF_b$, Corollary~\ref{cor:worst-case-error_dual} gives
\[ (e^{\wor}(H_{\alpha}, P))^2 = \sum_{\bsk,\bsl \in P^{\perp}\setminus \{\bszero\}}\widehat{K_\alpha}(\bsk,\bsl), \]
and
\[ \int_{[0,1]^s}(e^{\wor}(H_{\alpha}, P\oplus \bsdelta))^2\rd \bsdelta = \sum_{\bsk\in P^{\perp}\setminus \{\bszero\}}\widehat{K_\alpha}(\bsk,\bsk). \]
Similarly to Lemma~\ref{lem:walsh_L2disc}, the following result is known.
\begin{lemma}\label{lem:walsh_sobolev}
For $\bsk,\bsl\in \NN_0^s$, we have 
\begin{enumerate}
\item $|\widehat{K_{\alpha}}(\bsk,\bsl)| \leq C_{\alpha,b}^s b^{-\mu_\alpha(\bsk)-\mu_\alpha(\bsl)}$ with 
\[ C_{\alpha,b} =\max_{1\leq \nu\leq \alpha} \left\{ \sum_{\tau=\nu}^{\alpha}\frac{(C_{\tau,b})^2}{b^{2(\tau-\nu)}}+\frac{2C_{2\alpha,b}}{b^{2(\alpha-\nu)}}\right\}, \]
where
\[ C_{1,b}=\frac{1}{2\sin(\pi/b)}\quad \text{and} \quad C_{\tau,b}=\frac{(1+1/b+1/(b(b+1)))^{\tau-2}}{(2\sin(\pi/b))^{\tau}}\quad \text{for $\tau\geq 2$.} \]
\item $\widehat{K_{\alpha}}(\bsk,\bsl)=0$ if $(\bsk,\bsl)\in T_{\geq 2\alpha+1}$.
\end{enumerate}
\end{lemma}
\noindent The first assertion on the decay of the Walsh coefficients was shown in \cite{BD09}, while the second assertion of the sparsity of the Walsh coefficients was shown in \cite{GSY3}. Regarding the first assertion, we also refer to more recent works \cite{SY16,Y17} which introduce different approaches from the one by Dick \cite{Dic07,Dic08,Dic09} for evaluating the Walsh coefficients. In many cases, one may obtain smaller constants $C_{\alpha,b}$.

\begin{remark}
A lower bound on the worst-case error of order $(\log N)^{(s-1)/2}/N^{\alpha}$, which holds for any (non-linear/adaptive) quadrature rule based on $N$ function evaluations, can be proven by adapting the bump function technique from \cite{Bak59}. We also refer to \cite[Theorem~4 and Appendix]{DNP14} whose result directly applies to the present problem. As we shall show, higher order digital nets and sequences achieve this best possible order exactly.
\end{remark}

\subsection{Optimal order error bounds}

First we discuss what happens if we exploit only the decay of the Walsh coefficients. A similar result has been already proven in \cite[Theorem~30]{BD09} but with the slightly worse exponent of the $\log N$ term.
\begin{theorem}\label{thm:sub-optimal_error}
Let $\alpha\geq 2$. Let $P$ be an order $\alpha$ digital $(t_{\alpha},m,s)$-net over $\FF_b$. Then there exists a constant $E_{\alpha,b,s,t_{\alpha}}^{(1)}>0$ such that the following holds:
\[ \int_{[0,1]^s}(e^{\wor}(H_{\alpha}, P\oplus \bsdelta))^2\rd \bsdelta \leq E_{\alpha,b,s,t_{\alpha}}^{(1)}\frac{m^{s\alpha -1}}{b^{2\alpha m}}. \]
\end{theorem}
\begin{proof}
Using the first assertion of Lemma~\ref{lem:walsh_sobolev} and the second assertion of Lemma~\ref{lem:HO_digital-net-card} (with $\lambda=2$), we have
\begin{align*}
& \int_{[0,1]^s}(e^{\wor}(H_{\alpha}, P\oplus \bsdelta))^2\rd \bsdelta \\
& = \sum_{\bsk\in P^{\perp}\setminus \{\bszero\}}\widehat{K_\alpha}(\bsk,\bsk) \leq C_{\alpha,b}^s\sum_{\bsk\in P^{\perp}\setminus \{\bszero\}}b^{-2\mu_\alpha(\bsk)} \\
& \leq C_{\alpha,b}^s2^{s\alpha-1}(b-1)^{s\alpha }\frac{(\mu_\alpha(P^{\perp}))^{s\alpha -1}}{b^{2 \mu_\alpha(P^{\perp})}}\sum_{z=0}^{\infty}b^{(1/\alpha-2) z}(z+2)^{s\alpha -1}.
\end{align*}
By considering the equality $\mu_{\alpha}(P^{\perp}) = \alpha m-t_{\alpha}+1$, we prove the theorem.
\end{proof}
\noindent
Again the exponent $(s\alpha-1)$ of the numerator of the theorem stems from the inequality
\[ \left| \left\{ \bsk\in P^{\perp}\setminus \{\bszero\}\mid \mu_\alpha(\bsk)=z \right\} \right| \leq  (b-1)^{s\alpha }b^{(z-\mu_\alpha(P^{\perp}))/\alpha}(z+2)^{\underline{s\alpha -1}} \]
given in the first assertion of Lemma~\ref{lem:HO_digital-net-card}. It seems hard to fundamentally improve this bound. Thus, even before exploiting the sparsity of the Walsh coefficients, there is a difficulty in obtaining the best possible order of the shift-averaged worst-case error. 

To overcome this issue, let us go back to the bound given in the first assertion of Lemma~\ref{lem:digital-net-card}:
\[ \left| \left\{ \bsk\in P^{\perp}\setminus \{\bszero\}\mid \mu_1(\bsk)=z \right\} \right| \leq b^{z-\mu_1(P^{\perp})+1}(z+1)^{\underline{s-1}} . \]
As we discussed in Subsection~\ref{subsec:optimal-L_2-bounds}, this gives the optimal order of the shift-averaged $L_2$-discrepancy. Considering that the best possible exponent of the $\log N$ term for the present integration problem is the same as that of the $L_2$-discrepancy, which is $(s-1)/2$, one idea is to switch the weight function from $\mu_{\alpha}$ to $\mu_1$ in the error analysis. The following interpolation inequality was shown in \cite{GSY1} to realize this:
\begin{lemma}\label{lem:weight_jensen}
Let $\alpha,\beta\in \NN$ with $1<\alpha<\beta$. For any $\bsk\in \NN_0^s$ we have
\[ \mu_{\alpha}(\bsk) \geq \frac{\alpha-1}{\beta-1}\mu_{\beta}(\bsk) + \frac{\beta-\alpha}{\beta-1}\mu_1(\bsk). \]
\end{lemma}

By using this inequality together with the propagation rule of higher order digital nets (Lemma~\ref{lem:propagation}), we can improve the order of the shift-averaged worst-case error to best possible. The following result is from \cite{GSY1}.

\begin{theorem}\label{thm:optimal_error_gsy1}
Let $\alpha\geq 2$. Let $P$ be an order $2 \alpha$ digital $(t_{2\alpha},m,s)$-net over $\FF_b$. Then there exists a constant $E_{\alpha,b,s,t_{2\alpha}}^{(2)}>0$ such that the following holds:
\[ \int_{[0,1]^s}(e^{\wor}(H_{\alpha}, P\oplus \bsdelta))^2\rd \bsdelta \leq E_{\alpha,b,s,t_{2\alpha}}^{(2)}\frac{m^{s -1}}{b^{2\alpha m}}. \]
\end{theorem}
\begin{proof}
We write $A=(\alpha-1)/(2\alpha-1)$ and $B=\alpha/(2\alpha-1)$. Using the first assertion of Lemma~\ref{lem:walsh_sobolev}, Lemma~\ref{lem:weight_jensen} (with $\beta=2\alpha$) and the second assertion of Lemma~\ref{lem:digital-net-card} (with $\lambda=2B$), we have
\begin{align*}
\int_{[0,1]^s}(e^{\wor}(H_{\alpha}, P\oplus \bsdelta))^2\rd \bsdelta & = \sum_{\bsk\in P^{\perp}\setminus \{\bszero\}}\widehat{K_\alpha}(\bsk,\bsk) \leq C_{\alpha,b}^s\sum_{\bsk\in P^{\perp}\setminus \{\bszero\}}b^{-2\mu_\alpha(\bsk)} \\
& \leq C_{\alpha,b}^s\sum_{\bsk\in P^{\perp}\setminus \{\bszero\}}b^{-2A\mu_{2\alpha}(\bsk)-2B\mu_1(\bsk)} \\
& \leq C_{\alpha,b}^sb^{-2A\mu_{2\alpha}(P^{\perp})}\sum_{\bsk\in P^{\perp}\setminus \{\bszero\}}b^{-2B\mu_1(\bsk)} \\
& \leq C_{\alpha,b}^s2^{s-1}b^{2B}\frac{(\mu_1(P^\perp))^{s-1}}{b^{2A\mu_{2\alpha}(P^{\perp})+2B \mu_1(P^\perp)}} \sum_{z=1}^\infty b^{(1-2B)z} z^{s-1}.
\end{align*}
Since $2B-1=1/(2\alpha-1)>0$, the last sum over $z$ is finite. Using Lemma~\ref{lem:propagation} we have
\[ \mu_1(P^{\perp}) \geq m-\lceil t_{2\alpha}/(2\alpha)\rceil +1,  \]
and so 
\begin{align*}
2A\mu_{2\alpha}(P^{\perp})+2B\mu_1(P^{\perp}) & \geq 2A(2\alpha m -t_{2\alpha}+1)+2B(m-\lceil t_{2\alpha}/(2\alpha)\rceil +1) \\
& = 2\alpha m-2A(t_{2\alpha}-1)-2B(\lceil t_{2\alpha}/(2\alpha)\rceil -1),
\end{align*}
from which the result of the theorem follows.
\end{proof}
\noindent
It is important to recall that we consider order $2\alpha$ digital nets in this theorem, instead of order $\alpha$ digital nets as in Theorem~\ref{thm:sub-optimal_error}. The order $2\alpha$ is required here to ensure the finiteness of the sum 
\[ \sum_{z=1}^\infty b^{(1-2B)z} z^{s-1}. \]

Finally, combining the idea of switching the weight function (Theorem~\ref{thm:optimal_error_gsy1}) with the idea of exploiting both the decay and the sparsity of the Walsh coefficients simultaneously (Theorem~\ref{thm:L_2-disc}), we arrive at a deterministic counterpart of Theorem~\ref{thm:optimal_error_gsy1}, proven in \cite{GSY3}.
\begin{theorem}\label{thm:optimal_error_gsy3}
Let $\alpha\geq 2$. Let $P$ be an order $(2\alpha+1)$ digital $(t_{2\alpha+1},m,s)$-net over $\FF_b$. Then there exists a constant $E_{\alpha,b,s,t_{2\alpha+1}}^{(3)}>0$ such that the following holds:
\[ (e^{\wor}(H_{\alpha}, P))^2 \leq E_{\alpha,b,s,t_{2\alpha+1}}^{(3)}\frac{m^{s-1}}{b^{2\alpha m}}. \]
\end{theorem}
\begin{proof}
We write $A=(\alpha-1)/(2\alpha)$ and $B=(\alpha+1)/(2\alpha)$. Using Lemmas~\ref{lem:walsh_sobolev}, \ref{lem:weight_jensen} (with $\beta=2\alpha+1$) and \ref{lem:HO_digital-net-card2} (with $\alpha$ replaced by $2\alpha+1$) in order and then applying the change of variables $z\to z+2\mu_1(P^{\perp})$, we have
\begin{align*}
&(e^{\wor}(H_{\alpha}, P))^2 \\
& = \sum_{\bsk,\bsl \in P^{\perp}\setminus \{\bszero\}}\widehat{K_\alpha}(\bsk,\bsl) \leq C_{\alpha,b}^s \sum_{\substack{\bsk,\bsl\in P^{\perp}\setminus \{\bszero\}\\ (\bsk,\bsl)\notin T_{\geq 2\alpha+1}}}b^{-\mu_{\alpha}(\bsk)-\mu_{\alpha}(\bsl)} \\
& \leq C_{\alpha,b}^sb^{-2A\mu_{2\alpha+1}(P^{\perp})} \sum_{\substack{\bsk,\bsl\in P^{\perp}\setminus \{\bszero\}\\ (\bsk,\bsl)\notin T_{\geq 2\alpha+1}}}b^{-B(\mu_1(\bsk)+\mu_1(\bsl))} \\
& = C_{\alpha,b}^sb^{-2A\mu_{2\alpha+1}(P^{\perp})} \sum_{z=2\mu_1(P^{\perp})}^{\infty}b^{-Bz} \\
& \quad \times \left| \{ (\bsk,\bsl)\in (P^{\perp}\setminus \{\bszero\})^2\mid \mu_1(\bsk)+\mu_1(\bsl)=z, (\bsk,\bsl)\notin T_{\geq 2\alpha+1} \}\right| \\
& \leq B_{2\alpha+1,b,s,t_{2\alpha+1}}C_{\alpha,b}^sb^{-2A\mu_{2\alpha+1}(P^{\perp})} \sum_{z=2\mu_1(P^{\perp})}^{\infty}b^{-Bz}(z-2\mu_1(P^{\perp}))^{2s\alpha+1}z^{s-1}b^{(z-2\mu_1(P^{\perp}))/2} \\
& \leq B_{2\alpha+1,b,s,t_{2\alpha+1}}2^{s-1}C_{\alpha,b}^s\frac{(\mu_1(P^{\perp}))^{s-1}}{b^{2A\mu_{2\alpha+1}(P^{\perp})+2B\mu_1(P^{\perp})}}\sum_{z=0}^{\infty}b^{(1/2-B)z}z^{(2\alpha+1)s}.
\end{align*}
Since $B-1/2=1/(2\alpha)>0$, the last sum over $z$ is finite. Using Lemma~\ref{lem:propagation} we have
\[ \mu_1(P^{\perp}) \geq m-\lceil t_{2\alpha+1}/(2\alpha+1)\rceil +1,  \]
and so 
\begin{align*}
& 2A\mu_{2\alpha+1}(P^{\perp})+2B\mu_1(P^{\perp}) \\
& \geq 2A((2\alpha+1) m -t_{2\alpha+1}+1)+2B(m-\lceil t_{2\alpha+1}/(2\alpha+1)\rceil +1) \\
& = 2\alpha m-2A(t_{2\alpha+1}-1)-2B(\lceil t_{2\alpha+1}/(2\alpha+1)\rceil -1),
\end{align*}
from which the result of the theorem follows.
\end{proof}
\noindent
The error bound shown in Theorem~\ref{thm:optimal_error_gsy3} also applies to the first $b^m$ points of an order $(2\alpha+1)$ digital $(t_{2\alpha+1},s)$-sequence over $\FF_b$ if $m\geq t_{2\alpha+1}/(2\alpha+1)$, since they can be identified with an order $(2\alpha+1)$ digital $(t_{2\alpha+1},m,s)$-net over $\FF_b$.

\section{Conclusions and outlook}
In this article we have reviewed some of recent results on HoQMC methods with the particular aim to provide a unified picture on how the Walsh analysis enables these developments. The challenge in analyzing either the $L_2$-discrepancy or the worst-case error in a reproducing kernel Hilbert space is that we have to deal with the double sum of the Walsh coefficients. Considering the shift-averaged worst-case error instead, the problem becomes much easier in many cases, but of course, the outcome will remain probabilistic. To obtain a deterministic counterpart of such probabilistic result, exploiting the decay and the sparsity of the Walsh coefficients simultaneously seems to be a reasonable strategy to attack the problem. In fact, as we have seen, both the optimal order $L_2$-discrepancy bound in \cite{DP14} and the optimal order quadrature error bound in \cite{GSY3} are obtained by employing this strategy.

Looking into the future, there are some possible directions for further research as raised below.
\begin{enumerate}
\item \textbf{Choice of an orthonormal basis:} The system of Walsh functions fits quite well with digital nets as emphasized in this article, but is not the only choice. Indeed, recent papers \cite{M15,DHMP17a,DHMP17b,BM18} use the system of \emph{Haar functions} and succeed in generalizing or extending the result of \cite{DP14}. For instance, it was proven in \cite{DHMP17a} that order 2 (instead of order 5) digital sequences achieve the best possible order of $L_2$-discrepancy. Also, prior to the works of \cite{GSY1,GSY2,GSY3}, Hinrichs et al.\ used the system of \emph{Faber functions} to analyze the worst-case error of order 2 digital nets for different function spaces \cite{HMOT15}.  As natural questions from the current status, one may ask ``Can we get better results in numerical integration problems by using the system of Haar/Faber functions? Can we lower the necessary order of digital nets and sequences from $2\alpha+1$ to achieve the best possible error rate?'' We do not have any progress on these questions so far.
\item \textbf{Alternative construction scheme:} Higher order digital nets and sequences can be explicitly constructed through the digit interlacing function (Definition~\ref{def:interlacing} and Lemma~\ref{lem:ho_construction}). Quite recently, it has been shown in \cite{DGY19} that Richardson extrapolation plays an alternative role to the digit interlacing when the class of underlying point sets are restricted to polynomial lattice point sets. Also, the paper \cite{G2} proposed a different usage of Richardson extrapolation in the context of HoQMC methods. It is desirable to have more different options for explicit construction of higher order digital nets and sequences. 
\item \textbf{Universality for various function classes:} One major drawback of higher order digital nets and sequences is that we need to construct point sets or sequences depending on dominating mixed smoothness of the considered function space. Propagation rule (Lemma~\ref{lem:propagation}) says that, once we construct point sets or sequences which work for a certain smoothness $\beta$, they also work for any smaller smoothness $1\leq \alpha <\beta$ but not for larger smoothness. Ideally what we want in practice is point sets or sequences which work for all ranges of smoothness. One straightforward idea is to construct \emph{infinite order} digital nets and sequences and then study their propagation rule. In this line of research, we refer to \cite{SY16,Y17} for the Walsh analysis of infinitely many times differentiable functions, and furthermore, to \cite{MSM14,S14,S15,S16,DGSY17,MOY18} for the relevant literature.
\end{enumerate}

\section*{Acknowledgments}
This article is based on the talk which the first author gave during the discrepancy workshop of RICAM special semester ``Multivariate Algorithms and their Foundations in Number Theory''. He would like to thank the organizers of the workshop, Dmitriy Bilyk, Josef Dick, and Friedrich Pillichshammer, for their kind invitation. The authors sincerely acknowledge their colleague Takehito Yoshiki. Some of the outcomes brought from the collaboration with him play a central part of this article. The work of T.~G. is supported by JSPS Grant-in-Aid for Young Scientists No.~15K20964. The work of K.~S. is supported by Grant-in-Aid for JSPS Fellows No. 17J00466.

\end{document}